\documentclass[utf8]{frontiersSCNS}
\usepackage{amsmath,amssymb, amsthm}

\usepackage{multirow}
\usepackage{amsfonts}
\usepackage{layout}
\usepackage{url}
\usepackage{color}
\usepackage{xcolor}
\usepackage{graphicx}
\usepackage{algorithmic}
\usepackage{algorithm}
\usepackage{verbatim}
\usepackage{epsfig}
\usepackage{tikz}
\usepackage{mathtools}

\PassOptionsToPackage{normalem}{ulem}
\usepackage{ulem}

\usepackage{cases}

\providecolor{added}{rgb}{0,0,1}
\providecolor{deleted}{rgb}{1,0,0}
\definecolor{orange}{RGB}{255,127,0}

\newcommand{\eqdef}{\stackrel{\text{def}}{=}}

\newcommand{\R}{\mathbb{R}}

\newcommand{\E}{\mathbf{E}}

\newcommand{\vc}[2]{#1^{(#2)}}

\newcommand{\m}[1]{~\mbox{#1}~}

\newcommand{\norm}[1]{\|#1\|}

\newtheorem{lemma}{Lemma}

\newtheorem{theorem}{Theorem}

\newtheorem{remark}{Remark}
\newtheorem{corollary}{Corollary}

\theoremstyle{plain}

\theoremstyle{definition}

\newtheorem{definition}{Definition}

\newcommand*{\starnr}{\stepcounter{equation}\tag{\theequation}}

\newcommand{\ttt}{^{(t)}}

\usepackage{url,lineno,microtype}
\usepackage[onehalfspacing]{setspace}
\usepackage{silence}
\usepackage{xcolor}

\def\keyFont{\fontsize{8}{11}\helveticabold }
\def\firstAuthorLast{Xi He {et~al.}} %use et al only if is more than 1 author
\def\Authors{Xi He\,$^{1}$, Rachael Tappenden$^{2}$
and Martin Tak\'a\v{c}\,$^{1}$}
\def\Address{$^{1}$ Industrial and Systems Engineering, Lehigh University, PA, USA\\ $^{2}$ School of Mathematics and Statistics, University of Canterbury, NZ}
\def\corrAuthor{Martin Tak\'a\v{c}}

\def\corrEmail{Takac.MT@gmail.com}

\begin{document}
\onecolumn
\firstpage{1}

\title[Dual Free Adaptive Minibatch SDCA for Empirical Risk Minimization]{Dual Free Adaptive Minibatch SDCA \\for Empirical Risk Minimization}

\author[\firstAuthorLast ]{\Authors} %This field will be automatically populated
\address{} %This field will be automatically populated
\correspondance{} %This field will be automatically populated

\extraAuth{}

\maketitle

\begin{abstract}
In this paper we develop an adaptive dual free Stochastic Dual Coordinate Ascent (adfSDCA) algorithm for regularized empirical risk minimization problems. This is motivated by the recent work on dual free SDCA of \cite{shalev2016sdca}.
 The novelty of our approach is that the coordinates to update at each iteration are selected non-uniformly from an adaptive probability distribution, and this extends the previously mentioned work which only allowed for a uniform selection of ``dual" coordinates from a fixed probability distribution.

We describe an efficient iterative procedure for generating the non-uniform samples, where the scheme selects the coordinate with the greatest potential to decrease the sub-optimality of the current iterate.
We also propose a heuristic variant of adfSDCA that is more aggressive than the standard approach. Furthermore, in order to utilize multi-core machines we consider a mini-batch adfSDCA algorithm and develop complexity results that guarantee the algorithm's convergence. The work is concluded with several numerical experiments to demonstrate the practical benefits of the proposed approach.

\tiny
 \keyFont{ \section{Keywords:} SDCA, Importance Sampling, Non-uniform Sampling, Mini-batch, Adaptive} %All article types: you may provide up to 8 keywords; at least 5 are mandatory.
\end{abstract}

\section{Introduction}
\label{sec: Introduction}
In this work we study the $\ell_2$-regularized Empirical Risk Minimization (ERM) problem, which is widely used in the field of  machine learning. The problem can be stated as follows. Given training examples $(x_1,y_1), \dots,  (x_n,y_n) \in \R^d \times \R$, loss functions $\phi_1, \dots, \phi_n: \R \rightarrow \R$ and a regularization parameter $\lambda > 0$, $\ell_2$-regularized ERM is an optimization problem of the form
\begin{equation}\label{Prob: L2EMR}
\tag{P}
\min_{w\in \R^d} P(w) := \frac{1}{n}\sum_{i=1}^n\phi_i(w^Tx_i) + \frac{\lambda}{2}\norm{w}^2,
\end{equation}
where the first term in the objective function is a \emph{data fitting term} and the second is a \emph{regularization term} that prevents over-fitting.

Many algorithms have been proposed to solve problem~\eqref{Prob: L2EMR} over the past few years, including SGD, \cite{shalev2011pegasos},
SVRG and S2GD, \cite{johnson2013accelerating,nitanda2014stochastic,konevcny2014ms2gd} and SAG/SAGA,
\cite{schmidt2013minimizing,defazio2014saga,roux2012stochastic}. However, another very popular approach to solving $\ell_2$-regularized ERM problems is to consider the following dual formulation
\begin{equation}
\tag{D} \label{Prob:dual}
\max_{\alpha \in \R^n} D(\alpha) :=
-\frac1n \sum_{i=1}^n \phi^*_i(-\alpha_i) -\frac{\lambda}{2} \|\frac1{\lambda n} X^T \alpha\|^2,
\end{equation}
where $X^T=[x_1,\dots, x_n]\in \R^{d \times n}$ is the data matrix and $\phi_i^*$ denotes the convex conjugate of $\phi_i$. The structure of the dual formulation \eqref{Prob:dual} makes it well suited to a multicore or distributed computational setting, and several algorithms have been developed to take advantage of this including \cite{hsieh2008dual}
\cite{
takavc2013mini,
jaggi2014communication,
ma2015adding,
takavc2015distributed,
qu2015quartz,
csiba2015stochastic,
zhang2015communication}.

One of the most popular methods for solving \eqref{Prob:dual} is Stochastic Dual Coordinate Ascent (SDCA). The algorithm proceeds as follows. At iteration $t$ of SDCA a coordinate $i\in\{1, \dots, n\}$ is chosen uniformly at random and the current iterate $\vc{\alpha}{t}$ is updated to $\vc{\alpha}{t+1}:=\vc{\alpha}{t}+  \delta^* e_i$, where $\delta^* = \arg\max_{\delta\in \R} D(\vc{\alpha}{t}+ \delta e_i)$. Much research has focused on analysing the theoretical complexity of SDCA under various assumptions imposed on the functions $\phi_i^*$, including the pioneering work of Nesterov in \cite{nesterov2012efficiency} and others including \cite{richtarik2014iteration,
tappenden2015complexity,
necoara2013efficient,
necoara2013parallel,
liu2015asynchronous,
takavc2015distributed,
takavc2013mini}.

A modification that has led to improvements in the practical performance of SDCA is the use of \emph{importance sampling} when selecting the coordinate to update. That is, rather than using uniform probabilities,  instead coordinate $i$ is sampled with an arbitrary probability $p_i$, see for example \cite{zhao2014stochastic,csiba2015stochastic}.

In many cases algorithms that employ non-uniform coordinate sampling outperform na\"ive uniform selection, and in some cases help to decrease the number of iterations needed to achieve a desired accuracy by several fold.

{\bf Notation and Assumptions.}
In this work we use the notation $[n]\eqdef \{1,\dots,n\}$, as well as the following assumption. For all $i\in [n]$, the loss function $\phi_i$ is $\tilde L_i$-smooth with $\tilde L_i>0$, i.e., for any given $\beta, \delta \in \R$, we have
\begin{equation}
    | \phi_i'(\beta) - \phi_i'(\beta+\delta)| \leq \tilde L_i|\delta |.
\end{equation}
In addition, it is simple to observe that the function $\phi_i(x_i^T \cdot  ): \R^d \to \R$ is $L_i$ smooth, i.e., $\forall w,\bar w\in \R^d$ and for all $i \in [n]$ there exists a constant  $L_i\leq \|x_i\|^2 \tilde L_i $  such that
\begin{equation}\label{ass: smooth}
    \|\nabla\phi_i( x_i^T w)-\nabla\phi_i(x_i^T\bar{w})\| \leq L_i\norm{w-\bar{w}}.
\end{equation}
We will use the notation
\begin{equation}\label{eq:LLtilde}
  L=\max_{1\leq i\leq n} L_i, \qquad \text{and} \qquad \tilde{L} = \max_{1\leq i\leq n} \tilde{L}_i.
\end{equation}
Throughout this work we let $\R_+$ denote the set of nonnegative real numbers and we let $\R_+^n$ denote the set of $n$-dimensional vectors with all components being real and nonnegative.

\subsection{Contributions}
\label{ssec: Contributions}

In this section the main contributions of this paper are summarized (not in order of significance).

{\bf Adaptive SDCA.}
We modify the dual free SDCA algorithm proposed in \cite{DBLP:journals/corr/Shalev-Shwartz15} to allow for the adaptive adjustment of probabilities and a non-uniform selection of coordinates. Note that the method is dual free, and hence in contrast to classical SDCA, where the update is defined by maximizing the dual objective \eqref{Prob:dual}, here we define the update slightly differently (see Section \ref{sec:algorithm_adfSDCA} for details).

Allowing non-uniform selection of coordinates from an adaptive probability distribution leads  to improvements in practical performance and the algorithm achieves a better complexity bound than in \cite{DBLP:journals/corr/Shalev-Shwartz15}. We show that the error after $T$ iterations is decreased by factor of
$\prod_{t=1}^T(1-\theta\ttt)\geq (1-\theta^*)^T$ on average, where $\theta^*$ is an uniformly lower bound for all $\theta\ttt$.
 Here $1-\theta\ttt \in (0,1)$ is a parameter that depends on the current iterate $\vc{\alpha}{t}$ and the nonuniform probability distribution. By changing the coordinate selection strategy from uniform selection to adaptive, each $1-\theta\ttt$ becomes smaller, which leads to an improvement in the convergence rate.

{\bf Non-uniform sampling procedure.}
Rather than using a uniform sampling of coordinates, which is the commonly used approach, here we propose the use of non-uniform sampling from an adaptive probability distribution. With this novel sampling strategy, we are able to generate non-uniform non-overlapping and proper (see Section \ref{sec:mini_batch_adfsdca}) samplings for arbitrary marginal distributions under only one mild assumptions. Indeed, we show that without the assumption, there is no such non-uniform sampling strategy. We also extend our sampling strategy to allow the selection of mini-batches.

 {\bf Better convergence and complexity results.}
By utilizing an adaptive probabilities strategy, we can derive complexity results for our new algorithm that, for the case when every loss function is convex, depend only on the \emph{average} of the Lipschitz constants $L_i$. This improves upon the complexity theory developed in \cite{DBLP:journals/corr/Shalev-Shwartz15} (which uses a uniform sampling) and \cite{csiba2015primal} (which uses an arbitrary but fixed probability distribution), because the results in those works depend on the \emph{maximum} Lipschitz constant. Furthermore, even though adaptive probabilities are used here, we are still able to retain the very nice feature of the work in \cite{DBLP:journals/corr/Shalev-Shwartz15}, and show that the variance of the update naturally goes to zero as the iterates converge to the optimum without any additional computational effort or storage costs. Our adaptive probabilities SDCA method also comes with an improved bound on the variance of the update in terms of the sub-optimality of the current iterate.

 {\bf Practical aggressive variant.}
Following from the work of \cite{csiba2015stochastic}, we propose an efficient heuristic variant of adfSDCA. For adfSDCA the adaptive probabilities must be computed at every iteration (i.e., once a single coordinate has been selected), which can be computationally expensive. However, for our heuristic adfSDCA variant the (exact/true) adaptive probabilities are only computed once at the beginning of each epoch (where an epoch is one pass over the data/$n$ coordinate updates), and during that epoch, once a coordinate has been selected we simply reduce the probability associated with that coordinate so it is not selected again during that epoch. Intuitively this is reasonable because, after a coordinate has been updated the dual residue associated with that coordinate decreases and thus the probability of choosing this coordinate should also reduce. We show that in practice this heuristic adfSDCA variant converges and the computational effort required by this algorithm is lower than adfSDCA (see Sections~\ref{sec:heuristic_adfsdca} and \ref{sec:numerical_experiments}).

 {\bf Mini-batch variant.}
We extend the (serial) adfSDCA algorithm to incorporate a mini-batch scheme. The motivation for this approach is that there is a computational cost associated with generating the adaptive probabilities, so it is important to utilize them effectively. We develop a non-uniform mini-batch strategy that allows us to update multiple coordinates in one iteration, and the coordinates that are selected have high potential to decrease the sub-optimality of the current iterate. Further, we make use of ESO framework (Expected Separable Overapproximation) (see for example \cite{richtarik2012parallel}, \cite{qu2015quartz}) and present theoretical complexity results for mini-batch adfSDCA. In particular, for mini-batch adfSDCA used with batchsize $b$, we derive the optimal probabilities to use at each iteration, as well as the best step-size to use to guarantee speedup.

\subsection{Outline}
This paper is organized as follows. In Section~\ref{sec:algorithm_adfSDCA} we introduce our new Adaptive Dual Free SDCA algorithm (adfSDCA), and highlight its connection with a reduced variance SGD method. In Section~\ref{sec:convergence_Analysis} we provide theoretical convergence guarantees for adfSDCA in the case when all loss functions $\phi_i(\cdot)$ are convex, and also in the case when individual loss functions are allowed to be nonconvex but the average loss functions $\sum_{i=1}^n\phi_i(\cdot)$ is convex. Section~\ref{sec:heuristic_adfsdca} introduces a practical heuristic version of adfSDCA, and in Section~\ref{sec:mini_batch_adfsdca} we present a mini-batch adfSDCA algorithm and provide convergence guarantees for that method. Finally, we present the results of our numerical experiments in Section~\ref{sec:numerical_experiments}. Note that the proofs for all the theoretical results developed in this work are left to the appendix.

\section{The Adaptive Dual Free SDCA Algorithm} % (fold)
\label{sec:algorithm_adfSDCA}

In this section we describe the Adaptive Dual Free SDCA (adfSDCA) algorithm, which is motivated by the dual free SDCA algorithm proposed by \cite{DBLP:journals/corr/Shalev-Shwartz15}.
Note that in dual free SDCA two sequences of primal and dual iterates, $\{\vc{w}{t}\}_{t=0}^\infty$ and $\{\vc{\alpha}{t}\}_{t=0}^\infty$ respectively, are maintained. At every iteration of that algorithm, the variable updates are computed in such a way that the well known
primal-dual relational mapping holds; for every iteration $t$:
\begin{equation}\label{eq:walphaPDmap}
\vc{w}{t} = \frac{1}{\lambda n} {\sum}_{i=1}^n \vc{\alpha_i}{t} x_i.
\end{equation}
The dual residue is defined as follows.
\begin{definition}[Dual residue, \cite{csiba2015stochastic}]
    The dual residue $\kappa^{(t)} = (\kappa^{(t)}_1, \dots, \kappa^{(t)}_n)^T\in \R^n$ associated with $(w\ttt,\alpha\ttt)$ is given by:
    \begin{equation}\label{eq:dualresidue}
        \vc{\kappa_i}{t} \eqdef \vc{\alpha_i}{t} + \phi'_i(x_i^Tw^{(t)}).
    \end{equation}
\end{definition}

The Adaptive Dual Free SDCA algorithm is outlined in Algorithm~\ref{Alg: adfSDCA} and is described briefly now; a more detailed description (including a discussion of coordinate selection and how to generate appropriate selection rules) will follow.
An initial solution $\vc{\alpha}{0}$ is chosen, and then $\vc{w}{0}$ is defined via \eqref{eq:walphaPDmap}.
In each iteration of Algorithm~\ref{Alg: adfSDCA} the dual residue $\vc{\kappa}{t}$ is computed via \eqref{eq:dualresidue}, and this is used to generate a probability distribution $\vc{p}{t}$. Next, a coordinate $i \in [n]$ is selected (sampled) according to the generated probability distribution and a step is taken by updating the $i$th coordinate of $\alpha$ via
\begin{eqnarray}\label{alphaupdate}
\alpha_i^{(t+1)} = \alpha_i^{(t)} - \theta\ttt (p_i^{(t)})^{-1} \kappa^{(t)}_i.
\end{eqnarray}
Finally, the vector $w$ is also updated
\begin{eqnarray}\label{wupdate}
  w^{(t+1)} = w^{(t)} - \theta\ttt(n \lambda p_i^{(t)})^{-1} \kappa^{(t)}_ix_i,
\end{eqnarray}
and the process is repeated. Note that the updates to $\alpha$ and $w$ using the formulas \eqref{alphaupdate} and \eqref{wupdate} ensure that the equality \eqref{eq:walphaPDmap} is preserved.

Also note that the updates in \eqref{alphaupdate} and \eqref{wupdate} involve a step size parameter $\theta\ttt$, which will play an important role in our complexity results. The step size $\theta\ttt$ should be large so that good progress can be made, but it must also be small enough to ensure that the algorithm is guaranteed to converge. Indeed, in Section~\ref{sec:analysisCaseI} we will see that the choice of $\theta\ttt$ depends on the choice of probabilities used at iteration $t$, which in turn depend upon a particular function that is related to the suboptimality at iteration $t$.

\begin{algorithm}[ht]
    \caption{Adaptive Dual Free SDCA (adfSDCA)}
    \label{Alg: adfSDCA}
    \begin{algorithmic}[1]
        \STATE {\bf Input:} Data: $\{x_i, \phi_i\}_{i=1}^n$
        \STATE  {\bf Initialization:} Choose  $ \alpha^{(0)} \in \R^n$
        \STATE Set $w^{(0)} = \tfrac{1}{\lambda n}\sum_{i=1}^n \alpha_i^{(0)}x_i $
        \FOR {$t=0,1,2,\dots$}
        \STATE Calculate dual residual $\kappa^{(t)}_i = \phi'_i(x_i^Tw^{(t)}) + \alpha^{(t)}_i$, for all $i\in [n]$
        \STATE Generate adaptive probability distribution $p^{(t)} \sim \kappa^{(t)}$
        \STATE Sample coordinate $i$ according to $p^{(t)}$
        \STATE Set step-size $\theta\ttt \in (0, 1)$ as in \eqref{eq: optimal theta}
        \STATE {\bf Update:} $\alpha_i^{(t+1)} = \alpha_i^{(t)} - \theta\ttt (p_i^{(t)})^{-1} \kappa^{(t)}_i$
        \STATE \textbf{Update: }$w^{(t+1)} = w^{(t)} - \theta\ttt(n \lambda p_i^{(t)})^{-1} \kappa^{(t)}_ix_i$\label{step:wupdate}
        \ENDFOR
    \end{algorithmic}

\end{algorithm}

The dual residue $\kappa\ttt$ is informative and provides a useful way of monitoring suboptimality of the current solution $(w\ttt,\alpha\ttt)$. In particular, note that if $\kappa_i = 0$ for some coordinate $i$, then by \eqref{eq:dualresidue} $\alpha_i=-\phi'_i(w^Tx_i)$, and substituting $\kappa_i$ into \eqref{alphaupdate} and \eqref{wupdate} shows that $\alpha_i^{(t+1)} \gets \alpha_i\ttt$ and $w_i^{(t+1)} \gets w\ttt$, i.e., $\alpha$ and $w$ remain unchanged in that iteration. On the other hand, a large value of $|\kappa_i|$ (at some iteration $t$) indicates that a large step will be taken, which is anticipated to lead to good progress in terms of improvement in sub-optimality of current solution.

The probability distributions used in Algorithm~\ref{Alg: adfSDCA} adhere to the following definition.
\begin{definition} (Coherence, \cite{csiba2015stochastic})\label{Def: coherence}
Probability vector  $p\in \R^n$ is coherent with dual residue $\kappa\in \R^n$ if for any index $i$ in the support set of $\kappa$, denoted by $I_ \kappa := \{i \in [n]: \kappa_i \neq 0\}$, we have $p_i >0$.
When $i\notin I_ \kappa$ then $p_i = 0$. We use $p\sim \kappa$ to represent this coherent relation.
\end{definition}

\subsection{Adaptive dual free SDCA as a reduced variance SGD method.}
Reduced variance SGD methods have became very popular in the past few years, see for example
\cite{konevcny2013semi,johnson2013accelerating,roux2012stochastic,defazio2014saga}.
It is show in  \cite{DBLP:journals/corr/Shalev-Shwartz15}
that uniform dual free SDCA is an instance of a reduced variance SGD algorithm
(the variance of the stochastic gradient can be bounded by some measure of sub-optimality of the current iterate) and a similar result applies to adfSDCA in Algorithm~\ref{Alg: adfSDCA}. In particular, note that conditioned on $w^{(t-1)}$, we have
% {\color{red}\bf Xi, this should be conditional with respect to $\alpha$ not $w$, right???}
\begin{align}
\E[w^{(t)}| \alpha^{(t-1)}]
    &\overset{\eqref{wupdate}}{=} w^{(t-1)} - \frac{\theta^{(t-1)}}{\lambda}\sum_{i=1}^n\frac{ p_i}{n  p_i}   \left(\big(\nabla\phi_i(x_i^Tw^{(t-1)})+\alpha_i^{(t-1)}\big) x_i\right)\notag\\
    &\overset{\eqref{eq:walphaPDmap}}{=} w^{(t-1)} - \frac{\theta^{(t-1)}}{\lambda}\Big(\nabla\Big(\frac{1}{n}\sum_{i=1}^n\phi_i(x_i^Tw^{(t-1)})\Big) + \lambda w^{(t-1)}\Big) \notag\\
    &\overset{\eqref{Prob: L2EMR}}{=} w^{(t-1)} - \frac{\theta^{(t-1)}}{\lambda}\nabla P(w^{(t-1)}) .\label{eq:Expw}
\end{align}
Combining \eqref{wupdate} and \eqref{eq:Expw} gives
\begin{equation}\label{eq: nabla(P)}
    \E\left[\frac{1}{np_i}\kappa_i^{(t)}x_i
| \alpha^{(t-1)}
    \right] = \nabla P(w^{(t)}),
\end{equation}
which implies that $\displaystyle \tfrac{1}{np_i}\kappa_i^{(t)}x_i$ is an unbiased estimator of $\nabla P(w^{(t)})$. Therefore, Algorithm \ref{Alg: adfSDCA} is eventually a variant of the Stochastic Gradient Descent method. However, we can prove (see later) that the variance of the update goes to zero as the iterates converge to an optimum, which is not true for vanilla Stochastic Gradient Descent.

\section{Convergence Analysis}\label{sec:convergence_Analysis}
In this section we state the main convergence results for adfSDCA (Algorithm~\ref{Alg: adfSDCA}). The analysis is broken into two cases. In the first case it is assumed that each of the loss functions $\phi_i$ is convex. In the second case this assumption is relaxed slightly and it is only assumed that the average of the $\phi_i$'s is convex, i.e., individual functions $\phi_i(\cdot)$ for some (several) $i \in[n]$ are allowed to be nonconvex, as long as $\frac1n\sum_{j=1}^n \phi_j(\cdot)$ is convex. The proofs for all the results in this section can be found in the Appendix.

\subsection{Case I: All loss functions are convex}\label{sec:analysisCaseI}

Here we assume that $\phi_i$ is convex for all $i\in [n]$. Define the following parameter
\begin{equation}\label{eq:gamma}
  \gamma \eqdef \lambda\tilde{L},
\end{equation}
where $\tilde L$ is given in \eqref{eq:LLtilde}. It will also be convenient to define the following potential function. For all iterations $t\geq 0$,
\begin{equation}\label{eq:D}
  D\ttt \eqdef\tfrac{1}{n}\norm{\alpha\ttt-\alpha^*}^2 + \gamma  \norm{w\ttt-w^*}^2.
\end{equation}
The potential function \eqref{eq:D} plays a central role in the convergence theory presented in this work. It measures the distance from the optimum in both the primal and (pseudo) dual variables. Thus, our algorithm will generate iterates that reduce this suboptimality and therefore push the potential function toward zero.

Also define
\begin{equation}\label{eq:viQ}
  v_i \eqdef \norm{x_i}^2 \text{ for all } i \in [n].%, \qquad \text{and} \qquad v \eqdef \max_{1\leq i\leq n} v_i.
\end{equation}

We have the following result.
\begin{lemma}\label{thm: seperate convex}

% {\color{red}\bf Xi, this should be conditional expectation, right???}

Let $\tilde{L}$, $\kappa_i^{(t)}$, $\gamma$, $D\ttt$, and $v_i$ be as defined in \eqref{eq:LLtilde}, \eqref{eq:dualresidue}, \eqref{eq:gamma}, \eqref{eq:D} and \eqref{eq:viQ}, respectively. Suppose that $\phi_i$ is $\tilde{L}$-smooth and convex for all $i \in [n]$ and let $\theta\in(0,1)$. Then at every iteration $t\geq 0$ of Algorithm~\ref{Alg: adfSDCA}, a probability distribution $p\ttt$ that satisfies Definition~\ref{Def: coherence} is generated and
\begin{align}\label{eq: separate convex conclusion}
\E\big[D^{(t+1)} | \alpha\ttt \big] - (1- \theta)D\ttt  \leq \sum_{i=1}^n\left(-\frac{\theta}{n}\bigg(1- \frac{\theta}{p_i\ttt}\bigg)+\frac{\theta^2v_i \gamma}{n^2 \lambda^2 p_i\ttt}\right)(\kappa_i^{(t)})^2.
\end{align}
\end{lemma}
Note that if the right hand side of \eqref{eq: separate convex conclusion} is negative, then the potential function decreases (in expectation) in iteration $t$:
\begin{align}\label{eq:Dreduce}
    \E\big[D^{(t+1)} | \alpha\ttt\big] \leq (1- \theta)D\ttt.
\end{align}
The purpose of Algorithm~\ref{Alg: adfSDCA} is to generate iterates $(w\ttt,\alpha\ttt)$ such that the above holds. To guarantee negativity of the right hand term in \eqref{eq: separate convex conclusion}, or equivalently, to ensure that \eqref{eq:Dreduce} holds, consider the parameter $\theta$. Specifically, any $\theta$ that is less than the function $\Theta(\cdot, \cdot): \R_+^n \times \R_+^n \rightarrow \R$ defined as
\begin{equation}\label{eq:Theta}
    \Theta(\kappa, p) \eqdef \frac{n\lambda^2\sum_{i\in I_ \kappa}\kappa_i^2}{\sum_{i\in I_ \kappa}(n\lambda^2+v_i\gamma)p_i^{-1}\kappa_i^2},
\end{equation}
will ensure negativity of \eqref{eq: separate convex conclusion}. Moreover, the larger the value of $\theta$, the better progress Algorithm~\ref{Alg: adfSDCA} will make in terms of the reduction in $D\ttt$. The function $\Theta$ depends on the dual residue $\kappa$ and the probability distribution $p$. Maximizing this function w.r.t. $p$ will ensure that the largest possible value of $\theta$ can be used in Algorithm~\ref{Alg: adfSDCA}. Thus, we consider the following optimization problem:
\begin{align}\label{Prob: to derive best probabilities}
\max _{p\in \R^n_+, \sum_{i\in I_ \kappa} p_i=1} \quad&\Theta(\kappa, p).
\end{align}
One may naturally be wary of the additional computational cost incurred by solving the optimization problem in \eqref{Prob: to derive best probabilities} at every iteration. Fortunately, it turns out that there is an (inexpensive) closed form solution, as shown by the following Lemma.
\begin{lemma}\label{lem: optimal probabilities}
Let $\Theta(\kappa, p)$ be defined in \eqref{eq:Theta}. The optimal solution $p^*(\kappa)$ of \eqref{Prob: to derive best probabilities} is
\begin{equation}\label{eq: optimal probabilities}
p^*_i(\kappa) = \frac{\sqrt{v_i \gamma+n\lambda^2}|\kappa_i|}{\sum_{j\in I_ \kappa}\sqrt{v_j \gamma+n\lambda^2}|\kappa_j|}, \qquad \text{for all }\; i=1,\dots,n.
\end{equation}

The corresponding $\theta$ by using the optimal solution $p^*$ is
   \begin{align}\label{eq: optimal theta}
    \theta=\Theta(\kappa, p^*) =  \frac{n \lambda^2 \sum_{i\in I_ \kappa} \kappa_i^2}{(\sum_{i\in I_ \kappa}\sqrt{v_i \gamma + n \lambda^2}|\kappa_i|)^2}.
\end{align}
\end{lemma}
\begin{proof}
  This can be verified by deriving the KKT conditions of the optimization problem in \eqref{Prob: to derive best probabilities}. The details are moved to Appendix for brevity.
 \end{proof}

The results in \cite{csiba2015primal} are weaker because they require a fixed sampling distribution $p$ throughout all iterations. Here we allow adaptive sampling probabilities as in \eqref{eq: optimal probabilities}, which enables the algorithm to utilize the data information more effectively, and hence we have a better convergence rate. Furthermore, the optimal probabilities found in \cite{csiba2015stochastic} can be only applied to a quadratic loss function, whereas our results are more general because the optimal probabilities in \eqref{eq: optimal probabilities} can used whenever the loss functions are convex, or when individual loss functions are non-convex but the average of the loss functions is convex.

Before proceeding with the convergence theory we define several constants. Let
\begin{equation}\label{eq:C0}
  C_0 \eqdef \tfrac{1}{n}\norm{\alpha^{(0)}-\alpha^*}^2 + \gamma  \norm{w^{(0)}-w^*}^2,
\end{equation}
where $\gamma$ is defined in \eqref{eq:gamma}. Note that $C_0$ in \eqref{eq:C0} is equivalent to the value of the potential function \eqref{eq:D} at iteration $t=0$, i.e., $C_0 \equiv D^{(0)}$. Moreover, let
\begin{equation}\label{eq:MQ}
         M\eqdef Q\Big(1+ \frac{\gamma Q}{\lambda^2 n}\Big) \quad\mbox{ where }\quad Q\eqdef\frac{1}{n}\sum_{i=1}^n{\norm{x_i}^2}\overset{\eqref{eq:viQ}}{=}\frac{1}{n}\sum_{i=1}^n v_i.
     \end{equation}
Now we have the following theorem.
\begin{theorem}\label{cor: all convex}
Let $\tilde{L}$, $\kappa_i^{(t)}$, $\gamma$, $D\ttt$, $v_i$, $C_0$ and $Q$ be as defined in \eqref{eq:LLtilde}, \eqref{eq:dualresidue}, \eqref{eq:gamma}, \eqref{eq:D}, \eqref{eq:viQ}, \eqref{eq:C0} and \eqref{eq:MQ}, respectively. Suppose that $\phi_i$ is $\tilde{L}$-smooth and convex for all $i \in [n]$, let $\theta\ttt\in(0,1)$ be decided by \eqref{eq: optimal theta} for all $t\geq 0$
 and let $p^*$ be defined via \eqref{eq: optimal probabilities}. Then, setting $p^{(t)} = p^*$ at every iteration $t\geq 0$ of Algorithm~\ref{Alg: adfSDCA}, gives
\begin{align}
    \E[D^{(t+1)}| \alpha\ttt] \leq (1- \theta^*)D\ttt,
\end{align}
where
\begin{equation}\label{eq:thetastar}
  \theta^* \eqdef \frac{n \lambda^2}{\sum_{i=1}^n(v_i \gamma+n\lambda^2)} \leq \theta^{(t)}.
\end{equation}
Moreover, for $\epsilon>0$, if
\begin{equation}\label{eq:Tconvex}
    T\geq \bigg(n+\frac{\tilde LQ}{\lambda}\bigg)\log\left(\frac{(\lambda + L)C_0}{2\lambda \tilde{L}\epsilon}\right),
\end{equation}
then $\E[P(w^{(T)})-P(w^*)]\leq \epsilon$.
\end{theorem}

Similar to \cite{DBLP:journals/corr/Shalev-Shwartz15}, we have the following corollary which bounds the quantity $\E[ \norm{\tfrac{1}{np_i} \vc{\kappa_i}{t}x_i}^2 ] $ in terms of the sub-optimality of the points $\vc{\alpha}{t}$ and $w^{(t)}$ by using optimal probabilities.
\begin{corollary}\label{variance reduction}
    Let the conditions of Theorem~\ref{cor: all convex} hold. Then at every iteration $t\geq 0$ of Algorithm~\ref{Alg: adfSDCA},
\begin{equation*}
             \E\left[ \left\|\frac{\vc{\kappa_i}{t}x_i}{np_i}\right\|^2 |\alpha^{(t-1)}\right] \leq 2M(\E[\norm{\alpha^{(t)}-\alpha^*}^2|\alpha^{(t-1)}] + L\E[\norm{w^{(t)}-w^*}^2|\alpha^{(t-1)}]).
\end{equation*}
\end{corollary}
Note that Theorem~\ref{cor: all convex} can be used to show that both $\E[\norm{\alpha^{(t)}-\alpha^*}^2] $ and $\E[\norm{w^{(t)}-w^*}^2]$ go to zero as $e^{-\theta^*t}$. We can then show that $\E[ \norm{\tfrac{1}{np_i} \vc{\kappa_i}{t}x_i}^2 ]  \leq \epsilon$ as long as $t\geq \tilde{O}(\tfrac{1}{\theta^*}\log(\tfrac{1}{\epsilon}))$. Furthermore, we achieve the same variance reduction rate as shown in \cite{DBLP:journals/corr/Shalev-Shwartz15}, i.e., $\E[\norm{\frac{1}{np_i}\kappa_i^{(t)}x_i}^2]\sim \tilde{O}(\norm{\kappa^{(t)}}^2)$.

For the dual free SDCA algorithm in \cite{DBLP:journals/corr/Shalev-Shwartz15} where uniform sampling is adopted, the parameter $\theta$ should be set to at most $\min\tfrac{\lambda}{\lambda n + \tilde{L}}$, where $\tilde{L}\geq \max_iv_i \cdot  L$. However, from Corollary \ref{cor: all convex}, we know that this $\theta$ is smaller than $\theta^*$, so dual free SDCA will have a slower convergence rate than our algorithm. In \cite{csiba2015primal}, where they use a fixed probability distribution $p_i$ for sampling of coordinates, they must choose $\theta$ less than or equal to $\min_i\tfrac{p_in\lambda}{L_iv_i+n\lambda}$.
This is consistent with \cite{DBLP:journals/corr/Shalev-Shwartz15} where  $p_i=1/n$ for all $i\in [n]$. With respect to our adfSDCA Algorithm \ref{Alg: adfSDCA}, at any iteration $t$, we have that $\theta^{(t)}$ is greater than or equal to $\theta^*$, which again implies that our convergence results are better.

\subsection{Case II: The average of the loss functions is convex}\label{sec:analysisCaseII}

Here we follow the analysis in \cite{DBLP:journals/corr/Shalev-Shwartz15} and consider the case where individual loss functions $\phi_i(\cdot)$ for $i \in[n]$ are allowed to be nonconvex as long as the average $\frac1n\sum_{j=1}^n \phi_j(\cdot)$ is convex. First we define several parameters that are analogous to the ones used in Section~\ref{sec:analysisCaseI}. Let
\begin{equation}\label{eq:bargamma}
  \bar\gamma \eqdef \frac{1}{n}\sum_{i=1}^nL_i^2,
\end{equation}
where $L_i$ is given in \eqref{ass: smooth}, and define the following potential function. For all iterations $t\geq 0$, let
\begin{equation}\label{eq:barD}
  \bar D\ttt \eqdef\frac{1}{n}\norm{\alpha\ttt-\alpha^*}^2 + \bar\gamma  \norm{w\ttt-w^*}^2.
\end{equation}
We also define the following constants
\begin{equation}\label{eq:barC0}
  \bar C_0 \eqdef\frac{1}{n}\norm{\alpha^{(0)}-\alpha^*}^2 + \bar\gamma  \norm{w^{(0)}-w^*}^2,
\end{equation}
and
\begin{equation}\label{eq:barM}
         \bar M\eqdef Q\left(1+ \frac{\bar\gamma Q}{\lambda^2 n}\right).
     \end{equation}
Then we have the following theoretical results.
\begin{lemma}\label{thm: average convex}
Let $L_i$, $\kappa_i^{(t)}$, $\bar\gamma$, $\bar D\ttt$, and $v_i$ be as defined in \eqref{ass: smooth}, \eqref{eq:dualresidue}, \eqref{eq:bargamma}, \eqref{eq:barD} and \eqref{eq:viQ}, respectively. Suppose that every $\phi_i, i\in [n]$ is $L_i$-smooth and that the average of the $n$ loss functions $\tfrac{1}{n}\sum_{i=1}^n\phi_i(w^Tx_i)$ is convex. Let $\theta\in(0,1)$.
Then at every iteration $t\geq 0$ of Algorithm~\ref{Alg: adfSDCA}, a probability distribution $p\ttt$ that satisfies Definition~\ref{Def: coherence} is generated and
\begin{align}\label{eq: average convex conclusion}
\E[\bar D^{(t+1)}|\alpha\ttt] - (1- \theta)\bar D\ttt
\leq \sum_{i=1}^n\left(-\frac{\theta}{n}\bigg(1- \frac{\theta}{p_i\ttt}\bigg)+\frac{\theta^2v_i \bar\gamma}{n^2 \lambda^2 p_i\ttt}\right)(\kappa_i^{(t)})^2.
\end{align}
\end{lemma}

\begin{theorem}\label{cor: aver-convex}
Let $L$, $\kappa_i^{(t)}$, $\bar\gamma$ $\bar D\ttt$, $v_i$, and $\bar C_0$ be as defined in \eqref{eq:LLtilde}, \eqref{eq:dualresidue}, \eqref{eq:bargamma}, \eqref{eq:barD}, \eqref{eq:viQ}, and \eqref{eq:barC0} respectively. Suppose that every $\phi_i, i\in [n]$ is $L_i$-smooth and that the average of the $n$ loss functions $\tfrac{1}{n}\sum_{i=1}^n\phi_i(w^Tx_i)$ is convex. Let $\theta\ttt\in(0,1)$ using \eqref{eq: optimal theta} for all $t\geq 0$ and let $p^*$ be defined via \eqref{eq: optimal probabilities}. Then, setting $p^{(t)} = p^*$ at every iteration $t\geq 0$ of Algorithm~\ref{Alg: adfSDCA}, gives
\begin{align}
    \E\big[\bar D^{(t+1)}| \alpha\ttt\big] \leq (1- \theta^*)\bar D\ttt,
\end{align}
where
$$ \theta^*= \frac{n \lambda^2}{\sum_{i=1}^n(v_i \bar\gamma+n \lambda^2)} \leq \theta^{(t)}.$$
Furthermore, for $\epsilon>0$, if

\begin{equation}\label{eq:Tavconvex}
    T\geq \bigg(n + \frac{\bar \gamma Q}{\lambda^2}\bigg)\log\left(\frac{(\lambda + L)\bar C_0}{2\bar\gamma\epsilon}\right),
\end{equation}
then $\E[P(w^{(T)})-P(w^*)]\leq \epsilon$.
\end{theorem}

We remark that, $L_i\leq L$ for all $i\in [n]$, so $\bar\gamma \leq L^2$, which means that a conservative complexity bound is
\begin{equation*}
    T\geq \left(n + \frac{L^2 Q}{\lambda^2}\right)\log\left(\frac{(\lambda + L)\bar C_0}{2\bar\gamma\epsilon}\right).
\end{equation*}

We conclude this section with the following corollary.
\begin{corollary}\label{variance reduction_nonconvex}
    Let the conditions of Theorem~\ref{cor: aver-convex} hold and let $\bar M$ be defined in \eqref{eq:barM}. Then at every iteration $t\geq 0$ of Algorithm~\ref{Alg: adfSDCA},

\begin{equation*}
             \E\left[ \Big\|\frac{\vc{\kappa_i}{t}x_i}{np_i}\Big\|^2 |\alpha^{(t-1)}\right] \leq 2\bar M(\E[\norm{\alpha^{(t)}-\alpha^*}^2|\alpha^{(t-1)}] + L\E[\norm{w^{(t)}-w^*}^2|\alpha^{(t-1)}]).
\end{equation*}

\end{corollary}

\section{Heuristic adfSDCA} % (fold)
\label{sec:heuristic_adfsdca}

One of the disadvantages of Algorithm \ref{Alg: adfSDCA} is that it is necessary to update the entire probability distribution $p \sim \kappa$ at each iteration, i.e., every time a single coordinate is updated the probability distribution is also updated. Note that if the data are sparse and coordinate $i$ is sampled during iteration $t$, then, one need only update probabilities $p_j$ for which $x_j^Tx_i \neq 0$; unfortunately for some datasets this can still be expensive. In order to overcome this shortfall we follow the recent work in \cite{csiba2015stochastic} and present a heuristic algorithm that allows the probabilities to be updated less frequently and in a computationally inexpensive way. The process works as follows. At the beginning of each epoch the (full/exact) nonuniform probability distribution is computed, and this remains fixed for the next $n$ coordinate updates, i.e., it is fixed for the rest of that epoch. During that same epoch, if coordinate $i$ is sampled (and thus updated) the probability $p_i$ associated with that coordinate is reduced (it is shrunk by $p_i \gets p_i/s$). The intuition behind this procedure is that, if coordinate $i$ is updated then the dual residue $|\kappa_i|$ associated with that coordinate will decrease. Thus, there will be little benefit (in terms of reducing the sub-optimality of the current iterate) in sampling and updating that same coordinate $i$ again. To avoid choosing coordinate $i$ in the next iteration, we shrink the probability $p_i$ associated with it, i.e., we reduce the probability by a factor of $1/s$. Moreover, shrinking the coordinate is less computationally expensive than recomputing the full adaptive probability distribution from scratch, and so we anticipate a decrease in the overall running time if we use this heuristic strategy, compared with the standard adfSDCA algorithm. This procedure is stated formally in Algorithm~\ref{Alg: adfSDCA+}. Note that Algorithm~\ref{Alg: adfSDCA+} does not fit the theory established in Section~\ref{sec:convergence_Analysis}. Nonetheless, we have observed convergence in practice and a good numerical performance when using this strategy (see the numerical experiments in Section~\ref{sec:numerical_experiments}).

\begin{algorithm}[ht]
    \caption{Heuristic Adaptive Dual Free SDCA (adfSDCA+) }
    \label{Alg: adfSDCA+}
    \begin{algorithmic}[1]
        \STATE {\bf Input:} Data: $\{x_i, \phi_i\}_{i=1}^n$, probability shrink parameter $s$
        \STATE  {\bf Initialization:} Choose  $ \alpha^{(0)} \in \R^n$
        \STATE Set $w^{(0)} = \tfrac{1}{\lambda n}\sum_{i=1}^n \alpha_i^{(0)}x_i $
        \FOR {$t=0,1,2,\dots$}
        \IF {$\mod(t, n) == 0$}
        \STATE Calculate dual residue $\kappa^{(t)}_i = \phi'_i(x_i^Tw^{(t)}) + \alpha^{(t)}_i$, for all $i\in [n]$
        \STATE Generating adapted probabilities distribution $p^{(t)} \sim \kappa^{(t)}$
        \ENDIF
        \STATE Select coordinate $i$ from $[n]$ according to $p^{(t)}$
        \STATE Set step-size $\theta\ttt \in (0, 1)$ as in \eqref{eq: optimal theta}
        \STATE {\bf Update:} $\alpha_i^{(t+1)} = \alpha_i^{(t)} - \theta\ttt (p_i^{(t)})^{-1} \kappa^{(t)}_i$
        \STATE \textbf{Update: }$w^{(t+1)} = w^{(t)} - \theta\ttt(n \lambda p_i^{(t)})^{-1} \kappa^{(t)}_ix_i$
        \STATE \textbf{Update: } $p^{(t+1)}_i = p^{(t)}_i/s$
        \ENDFOR
    \end{algorithmic}

% {\color{red}\bf How are we choosing $\theta$?}

\end{algorithm}

\section{Mini-batch adfSDCA} % (fold)
\label{sec:mini_batch_adfsdca}

In this section we propose a mini-batch variant of Algorithm \ref{Alg: adfSDCA}. Before doing so, we stress that sampling a mini-batch non-uniformly is not easy. We first focus on the task of generating non-uniform random samples and then we will present our minibatch algorithm.

\subsection{Efficient single coordinate sampling} % (fold)
\label{sub:single coordinate sampling}
Before considering mini-batch sampling, we first show how to sample a single coordinate from a non-uniform distribution. Note that only discrete distributions are considered here.

There are multiple approaches that can be taken in this case. One na\"ive approach is to consider the Cumulative Distribution Function (CDF) of $p$, because a CDF can be computing in $O(n)$ time complexity and it also takes $O(n)$ time complexity to make a decision. One can also use a better data structure (e.g. a binary search tree) to reduce the decision cost to $O(\log n)$ time complexity, although the cost to set up the tree is $O(n \log n)$. Some more advanced approaches like the so-called alias method of \cite{kronmal1979alias} can be used to sample a single coordinate in only $O(1)$, i.e., sampling a single coordinate can be done in constant time but with a cost of $O(n)$ setup time. The alias method works based on the fact that any $n$-valued distribution can be written as a mixture of $n$ Bernoulli distributions.

In this paper we choose two sampling update strategies, one each for Algorithms~\ref{Alg: adfSDCA} and \ref{Alg: adfSDCA+}. For adfSDCA in Algorithm \ref{Alg: adfSDCA} the probability distribution must be recalculated at every iteration, so we use the alias method, which is highly efficient. The heuristic approach in Algorithm \ref{Alg: adfSDCA+} is a strategy that only alters the probability of a single coordinate (e.g. $p_i = p_i/s$) in each iteration. In this second case it is relatively expensive to use the alias method due to the linear time cost to update the alias structure, so instead we build a binary tree when the algorithm is initialized so that the update complexity reduces to $O(\log(n))$.

\subsection{Nonuniform Mini-batch Sampling}

Many randomized coordinate descent type algorithms utilize a sampling scheme that assigns every subset of $[n]$ a probability $p_S$, where $S\in 2^{[n]}$. In this section, we consider a particular type of sampling called a \emph{mini-batch} sampling that is defined as follows.

\begin{definition}\label{def: mini-batch sampling}
    A sampling $\hat{S}$ is called a mini-batch sampling, with batchsize $b$, consistent with the given marginal distribution $q:=(q_1, \dots, q_n)^T$, if the following conditions hold:
   \begin{enumerate}
    \item $|S|=b$; and
    \item $q_i\eqdef\sum_{i\in S, S\in \hat{S}}P(S) = bp_i$.
   \end{enumerate}
\end{definition}

Note that we study samplings  $\hat{S}$ that are \emph{non-uniform} since we allow $q_i$ to vary with $i$. The motivation to design such samplings arises from the fact that we wish to make use of the optimal probabilities that were studied in Section~\ref{sec:convergence_Analysis}.

We make several remarks about non-uniform mini-batch samplings below.
\begin{enumerate}
 \item For a given probability distribution $p$, one can derive a corresponding mini-batch sampling only if we have $p_i\leq \tfrac{1}{b}$ for all $i\in [n]$. This is obvious in the sense that $q_i = bp_i= \sum_{i\in S, S\in \hat{S}}P(S) \leq \sum_{S\in \hat{S}} P(S) = 1$.
 \item  For a given probability distribution $p$ and a batch size $b$, the mini-batch sampling may not be unique and it may not be proper, see for example \cite{richtarik2012parallel}. (A proper sampling is a sampling for which any  subset of size $b$ must have a \emph{positive} probability of being sampled.)
 \end{enumerate}

In Algorithm~\ref{Alg: Non-uniform Mini-batch sampling} we describe an approach that we used to generate a non-uniform mini-batch sampling of batchsize $b$ from a given marginal distribution $q$. Without loss of generality, we assume that the $q_i\in(0, 1)$ for $i\in [n]$ are sorted from largest to smallest.

\begin{algorithm}[ht]
    \caption{Non-uniform mini-batch sampling}
    \label{Alg: Non-uniform Mini-batch sampling}
    \begin{algorithmic}[1]
        \STATE {\bf Input:} Marginal distribution $q\in \R^n$ with $q_i\in (0, 1)$ $\forall i\in [n]$ and batchsize $b$ such that $\sum_{i=1}^nq_i=b$. Define $q_{n+1} = 0$
        \STATE  {\bf Output:} A mini-batch sampling $S$ (Definition \ref{def: mini-batch sampling})
        \STATE {\bf Initialization: } Index set $i, j\in \mathbb{N}^n$, and set $k= 1$.
        \FOR {$k= 1,\dots, n$}
        \STATE $i^k = \min_i\{i: p_i = q_b\}, j^k = \max_i\{i:p_i = q_b\}$
        \STATE {\bf Obtain }$r_k$:
        \begin{equation}
        r_k =\begin{cases}
              \min\Big{\{}\tfrac{j^k-i^k+1}{j^k-b}(q_{i^k-1}-q_b),
              %&\\\qquad
              \tfrac{j^k-i^k+1}{b-i^k+1}(q_{b}-q_{j^k+1})\Big{\}}, &i^k > 1\\
              \tfrac{j}{b}(q_b-q_{j^k+1}),& i^k = 1
        \end{cases}
        \end{equation}
        \STATE {\bf Update }$q_i$:
        \begin{equation}
            q_i = \begin{cases}
                q_i - r_k, &i\in [0, i^k-1],\\
                q_i - \tfrac{b-i^k+1}{j^k-i^k+1}r_k, &i\in[i^k, j^k]
            \end{cases}
        \end{equation}
        \STATE {\bf Terminate if } $q=0$, and set $m=k$
        \ENDFOR
        \STATE Select $K \in [m]$ randomly with discrete distribution $(r_1,\dots, r_{m})$
        \STATE Choose $b-i^K+1$ coordinates uniformly at random from $i^K$ to $j^K$, denote it by $W$
        \STATE $S = \{1,\dots, i^K-1\}\cup W$
    \end{algorithmic}
\end{algorithm}

We now state several facts about Algorithm \ref{Alg: Non-uniform Mini-batch sampling}.
\begin{enumerate}
    \item Algorithm~\ref{Alg: Non-uniform Mini-batch sampling} will terminate in at most $n$ iterations. This is because the update rules for $q_i$ (which depend on $r_k$ at each iteration), ensure that at least one $q_i$ will reduce to become equal to some $q_j < q_i$ (i.e., either  $q_{i^{k+1}-1} = q_b$ or $q_{j^{k+1}+1}=q_b$) and since there are $n$ coordinates in total, after at most $n$ iteration it must hold that $q_i=q_j$ for all $i,j\in[n]$. Note that if the algorithm begins with $q_i=q_j$ for all $i,j\in[n]$, which implies a uniform marginal distribution, the algorithm will terminated in a single step.
    \item For Algorithm~\ref{Alg: Non-uniform Mini-batch sampling} we must have $\sum_{i=1}^{m}r_i = 1$, where we assume that the algorithm terminates at iteration $m\in[1,n]$, since overall we have $\sum_{i=1}^{m}br_i = \sum_{i=1}^nq_i = b$.
    \item Algorithm~\ref{Alg: Non-uniform Mini-batch sampling} will always generate a proper sampling because when it terminates, the situation $p_i = p_j >0$, for all $i\neq j$, will always hold. Thus, any subset of size $b$ has a positive probability of being sampled.
    \item It can be shown that this algorithm works on an arbitrary given marginal probabilities as long as $q_i\in (0,1)$, for all $i\in[n]$.
\end{enumerate}

Figure \ref{fig : sampling demo} is a sample illustration of Algorithm \ref{Alg: Non-uniform Mini-batch sampling}, where we have a marginal distribution for $4$ coordinates given by $(0.8, 0.6, 0.4, 0.2)^T$ and we set the batchsize to be $b=2$. Then, the algorithm is run and finds $r$ to be $(0.2, 0.4, 0.4)^T$. Afterwards, with probability $r_1 = 0.2$, we will sample $2$-coordinates from $(1,2)$. With probability $r_2 = 0.4$, we will sample $2$-coordinates which has $(1)$ for sure and the other coordinate is chosen from $(2,3)$ uniformly at random and with probability $r_3 = 0.4$, we will sample $2$-coordinates from $(1,2,3,4)$ uniformly  at random.

Note that, here we only need to perform two kinds of operations. The first one is to sample a single coordinate from distribution $d$ (see Section~\ref{sub:single coordinate sampling}), and the second is to sample batches from a uniform distribution (see for example \cite{richtarik2012parallel}).
\begin{figure}[!ht]\centering
    \includegraphics[width=0.5\textwidth]{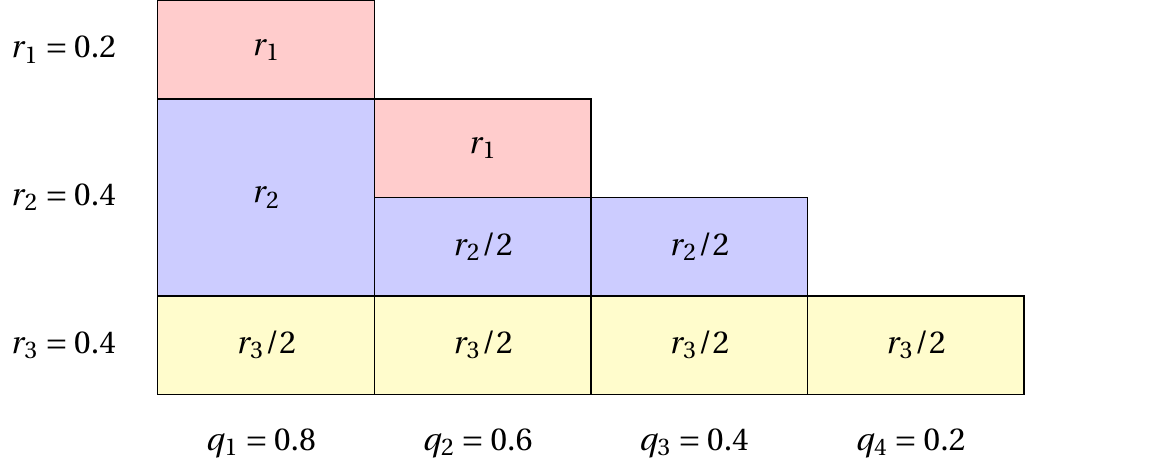}
    \caption{Toy demo illustrating how to obtain a non-uniform mini-batch sampling with batch size $b=2$ from $n=4$ coordinates.}
    \label{fig : sampling demo}
\end{figure}

\subsection{Mini-batch adfSDCA algorithm}
Here we describe a new adfSDCA algorithm that uses a mini-batch scheme. The algorithm is called mini-batch adfSDCA and is presented below as Algorithm \ref{Alg: mini-adfSDCA}.
\begin{algorithm}[ht]
    \caption{Mini-Batch adfSDCA}
    \label{Alg: mini-adfSDCA}
    \begin{algorithmic}[1]
            \STATE {\bf Input:} Data: $\{x_i, \phi_i\}_{i=1}^n$
        \STATE  {\bf Initialization:} Choose  $ \alpha^{(0)} \in \R^n$ and set batchsize $b$
        \FOR {$t = 0, 1, 2, \dots $}
                \STATE Calculate dual residue $\kappa^{(t)}_i = \phi'_i(x_i^Tw^{(t)}) + \alpha^{(t)}_i$, for all $i\in [n]$
        \STATE Generate the adaptive probability distribution $p^{(t)} \sim \kappa^{(t)}$
        \STATE Choose mini-batch $S\subset [n]$ of size $b$ according to probabilities distribution $p^{(t)}$
        \STATE Set step-size $\theta\ttt\in(0,1)$ as in \eqref{eq: optimal batch theta}
        \FOR {$i\in S$}
        \STATE {\bf Update:}
        $\alpha_i^{(t+1)} = \alpha_i^{(t)} -\theta\ttt (bp_i^{(t)})^{-1} \kappa^{(t)}_i$
        \ENDFOR
        \STATE \textbf{Update: } $w^{(t+1)} = w^{(t)} - \sum_{i\in S}\theta\ttt(n\lambda bp_i^{(t)})^{-1} \kappa^{(t)}_ix_i$
        \ENDFOR
    \end{algorithmic}
    % {\color{red}\bf again, do we have static $\theta$ or time depended?}
\end{algorithm}

Briefly, Algorithm~\ref{Alg: mini-adfSDCA} works as follows. At iteration $t$, adaptive probabilities are generated in the same way as for Algorithm~\ref{Alg: adfSDCA}. Then, instead of updating only one coordinate, a mini-batch $S$ of size $b\geq 1$ is chosen that is consistent with the adaptive probabilities. Next, the dual variables $\alpha_i^{(t)}, i\in S$ are updated, and finally the primal variable $w$ is updated according to the primal-dual relation \eqref{eq:walphaPDmap}.

In the next section we will provide a convergence guarantee for Algorithm~\ref{Alg: mini-adfSDCA}. As was discussed in Section \ref{sec:convergence_Analysis}, theoretical results are detailed under two different assumptions on the type of loss function: (i) all loss function are convex; and (ii) individual loss functions may be non-convex but the average over all loss functions is convex.

\subsection{Expected Separable Overapproximation}\label{sec:eso}
Here we make use of the Expected Separable Overapproximation (ESO) theory introduced in \cite{richtarik2012parallel} and further extended, for example, in \cite{qu2014coordinate}. The ESO definition is stated below.
\begin{definition}[Expected Separable Overapproximation, \cite{qu2014coordinate}]
    Let $\hat{S}$ be a sampling with marginal distribution $q=(q_1,\cdots,q_n)^T$. Then we say that the function $f$ admits a $v$-ESO with respect to the sampling $\hat{S}$ if $\forall x, h\in \R^n$, we have $v_1, \dots, v_n >0$, such that the following inequality holds
   $
        \E[f(x+h_{[\hat{S}]})] \leq f(x)
        + \sum_{i=1}^nq_i(\nabla_i f(x)h_i +  \tfrac{1}{2}v_ih_i^2).
    $
\end{definition}
\begin{remark}
  Note that, here we do not assume that $\hat{S}$ is a uniform sampling, i.e., we do not assume that $q_i =q_j$ for all $i, j \in [n]$.
\end{remark}

The ESO inequality is useful in this work because the parameter $v$ plays an important role when setting a suitable stepsize $\theta$ in our algorithm. Consequently, this also influences our complexity result, which depends on the sampling $\hat{S}$. For the proof of Theorem~\ref{thm: minibatch-all convex} (which will be stated in next subsection), the following is useful. Let $f(x) = \tfrac{1}{2}\norm{Ax}^2$, where $A = (x_1,\dots, x_n)$. We say that $f(x)$ admits a $v$-ESO if the following inequality holds
\begin{equation}\label{eq: mini-ESO}
        \E[\norm{Ah_{\hat{S}}}^2] \leq \sum_{i=1}^nv_iq_ih_i^2.
    \end{equation}
To derive the parameter $v$ we will make use of the following theorem.
\begin{theorem}[\cite{qu2014coordinate}]\label{thm: v-ESO}
Let $f$ satisfy the following assumption
$
    f(x+h) \leq f(x) + \langle \nabla f(x), h \rangle + \tfrac{1}{2}h^TA^TAh^T,
$
where $A$ is some matrix. Then, for a given sampling $\hat{S}$, $f$ admits a $v$-ESO, where $v$ is defined by
$    v_i = \min\{\lambda'(\mathbf{P}(\hat{S})), \lambda'(A^TA)\} \sum_{j=1}^mA_{ji}^2, i\in [n].
$
\end{theorem}
Here $\mathbf{P}(\hat{S})$ is called a sampling matrix (see \cite{richtarik2012parallel}) where element $p_{ij}$ is defined to be
$
    p_{ij} = \textstyle{\sum}_{\{i,j\}\in S, S\in \hat{S}} P(S).
$
For any matrix $M$, $\lambda'(M)$ denotes the maximal regularized eigenvalue of $M$, i.e.,
$
    \lambda'(M) = \max_{\norm{h}=1}\{h^TMh: \textstyle{\sum}_{i=1}^nM_{ii}h_i^2 \leq 1\}.
$
We may now apply Theorem \ref{thm: v-ESO} because $f(x) = \tfrac{1}{2}\norm{Ax}^2$ satisfies its assumption. Note that in our mini-batch setting, we have $P_{S\in \hat{S}}(|S| = b) = 1$, so we obtain $\lambda'(\mathbf{P}(\hat{S})) \leq  b$ (Theorem 4.1 in  \cite{qu2014coordinate}).
In terms of $\lambda'(A^TA)$, note that
$
    \lambda'(A^TA) = \lambda'(\sum_{j=1}^mx_jx_j^T)  \leq \max_j\lambda'(x_jx_j^T) = \max_j|J_j|,
$
where $|J_j|$ is number of non-zero elements of $x_j$ for each $j$. Then, a conservative choice from Theorem \ref{thm: v-ESO} that satisfies \eqref{eq: mini-ESO} is
\begin{equation}\label{ESO, v}
    v_i' = \min\{b, \max_j|J_j|\}\norm{x_i}^2, \qquad i\in [n].
\end{equation}

Now we are ready to give our complexity result for mini-batch adfSDCA (Algorithm \ref{Alg: mini-adfSDCA}). Note that we use the same notation as that established in Section \ref{sec:convergence_Analysis} and we also define
\begin{equation}\label{eq:Qprime}
  Q' \eqdef \frac1n \sum_{i=1}^n v_i'.
\end{equation}

\begin{theorem}\label{thm: minibatch-all convex}
Let $\tilde{L}$, $\kappa_i^{(t)}$, $\gamma$ $D\ttt$, $v_i'$, $C_0$ and $Q'$ be as defined in \eqref{eq:LLtilde}, \eqref{eq:dualresidue}, \eqref{eq:gamma}, \eqref{eq:D}, \eqref{ESO, v}, \eqref{eq:C0} and \eqref{eq:Qprime}, respectively. Suppose that $\phi_i$ is $L$-smooth and convex for all $i \in [n]$. Then, at every iteration $t\geq 0$ of Algorithm~\ref{Alg: mini-adfSDCA}, run with batchsize $b$ we have
    \begin{equation}
\E[D^{(t+1)}|\alpha\ttt] \leq (1- \theta^*)D^{(t)},
\end{equation}
where $\theta^* = \tfrac{n \lambda^2b}{\sum_{i=1}^n(v_i'\gamma + n \lambda^2)}$. Moreover, it follows that whenever
\begin{equation}
        T\geq \bigg(\frac{n}{b} + \frac{\tilde{L}Q'}{b\lambda}\bigg)\log\bigg(\frac{(\lambda +\tilde{L})C_0}{\lambda \tilde{L} \epsilon}\bigg),
    \end{equation}
we have that $\E[P(w^{(T)}-P(w^*))]\leq \epsilon$.
\end{theorem}
It is also possible to derive a complexity result in the case when the \emph{average} of the $n$ loss functions is convex. The theorem is stated now.

\begin{theorem}\label{thm: minibatch-seprate convex}
Let $L$, $\kappa_i^{(t)}$, $\bar\gamma$ $\bar D\ttt$, $v_i'$, $\bar C_0$ and $Q'$ be as defined in \eqref{eq:LLtilde}, \eqref{eq:dualresidue}, \eqref{eq:bargamma}, \eqref{eq:barD}, \eqref{ESO, v}, \eqref{eq:barC0} and \eqref{eq:Qprime} respectively. Suppose that every $\phi_i, i\in [n]$ is $L_i$-smooth and that the average of the $n$ loss functions $\tfrac{1}{n}\sum_{i=1}^n\phi_i(w^Tx_i)$ is convex. Then, at every iteration $t\geq 0$ of Algorithm~\ref{Alg: mini-adfSDCA}, run with batchsize $b$, we have
\begin{equation}
\E[\bar D^{(t+1)}|\alpha\ttt] \leq (1- \theta^*)\bar D^{(t)},
\end{equation}
where $\theta^* = \tfrac{n \lambda^2b}{\sum_{i=1}^n(v_i'\bar\gamma + n \lambda^2)}$. Moreover, it follows that whenever
\begin{equation}\label{eq:minibatchnonconvexT}
        T\geq \bigg(\frac{n}{b} + \frac{Q'\frac{1}{n}\textstyle{\sum}_{i=1}^nL_i^2}{b\lambda}\bigg)\log\bigg(\frac{(\lambda +\tilde{L})\bar C_0}{\bar \gamma \epsilon}\bigg),
    \end{equation}
we have that $\E[P(w^{(T)}-P(w^*)]\leq \epsilon$.
\end{theorem}
These theorems show that in worst case (by setting $b=1$), this mini-batch scheme shares the same complexity performance as the serial adfSDCA approach (recall Section \ref{sec:algorithm_adfSDCA}). However, when the batch-size $b$ is larger, Algorithm~\ref{Alg: mini-adfSDCA} converges in fewer iterations. This behaviour will be confirmed computationally in the numerical results given in Section \ref{sec:numerical_experiments}.

\section{Numerical experiments} % (fold)
\label{sec:numerical_experiments}

Here we present numerical experiments to demonstrate the practical performance of the adfSDCA algorithm. Throughout these experiments we used two loss functions, quadratic loss $ \phi_i(w^Tx_i) = \tfrac{1}{2}(w^Tx_i-y_i)^2$ and logistic loss $\phi_i(w^Tx_i) = \log(1+\exp(-y_iw^Tx_i))$. The experiments were run using datasets from the standard library of test problems (see \cite{Chang11} and \url{http://www.csie.ntu.edu.tw/~cjlin/libsvm}), as summarized in Table~\ref{tab:Datasets used in experiments}.\\

    \begin{table}[ht]
        \centering
        \begin{tabular}{l|c c c c}
            \hline
            \textbf{Dataset} & \textbf{\#samples} & \textbf{\#features} & \textbf{\#classes} & \textbf{sparsity}\\
            \hline
            \textbf{w8a} & $49,749$ & $300$ & $2$ & $3.91\%$\\
            \textbf{mushrooms} & $8,124$ & $112$ & $2$ & $18.8\%$\\
            \textbf{ijcnn1} & $49,990$ & $22$ & $2$ & $59.1\%$\\
             \textbf{rcv1} & $20,242$ & $47,237$ & $2$ & $0.16\%$\\
            \textbf{news20} & $19,996$ & $1,355,191$ & $2$ & $0.034\%$\\
            \hline
        \end{tabular}
        \caption{The datasets used in the numerical experiments \cite{Chang11}.}
        \label{tab:Datasets used in experiments}
    \end{table}

\subsection{Comparison for a variety of dfSDCA approaches}
In this section we compare the adfSDCA algorithm (Algorithm~\ref{Alg: adfSDCA}) with both dfSCDA, which is a uniform variant of adfSDCA described in \cite{DBLP:journals/corr/Shalev-Shwartz15}, and also with Prox-SDCA from \cite{shalev2014accelerated}. We also report results using Algorithm \ref{Alg: adfSDCA+}, which is a heuristic version of adfSDCA, used with several different shrinking parameters.

Figures \ref{Fig:evolution 1} and \ref{Fig:evolution 2} compare the evolution of the duality gap for the standard and heuristic variant of our adfSDCA algorithm with the two state-of-the-art algorithms dfSDCA and Prox-SDCA. For these problems both our algorithm variants out-perform the dfSDCA and Prox-SDCA algorithms. Note that this is consistent with our convergence analysis (recall Section \ref{sec:convergence_Analysis}). Now consider the adfSDCA\texttt{+} algorithm, which was tested using the parameter values $s=1,10,20$. It is clear that adfSDCA\texttt{+} with $s=1$ shows the worst performance, which is reasonable because in this case the algorithm only updates the sampling probabilities after each epoch; it is still better than dfSDCA since it utilizes the sub-optimality at the beginning of each epoch. On the other hand, there does not appear to be an obvious difference between adfSDCA\texttt{+} used with $s=10$ or $s=20$ with both variants performing similarly. We see that adfSDCA performs the best overall in terms of the number of passes through the data. However, in practice, even though adfSDCA\texttt{+} may need more passes through the data to obtain the same sub-optimality as adfSDCA, it requires less computational effort than adfSDCA.
\begin{figure}[ht]  \centering
    \includegraphics[width=0.3\textwidth]{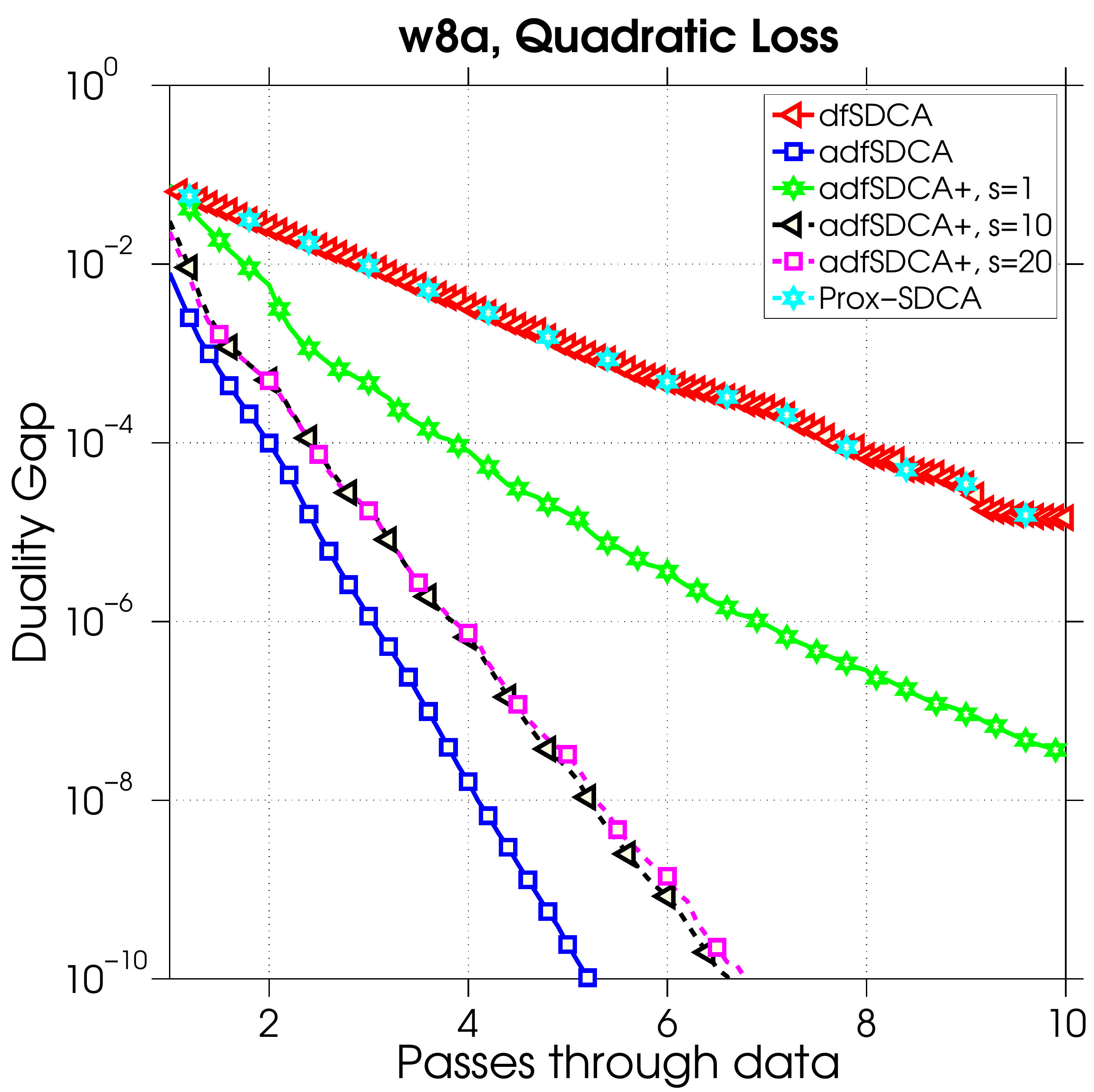}
    \includegraphics[width=0.3\textwidth]{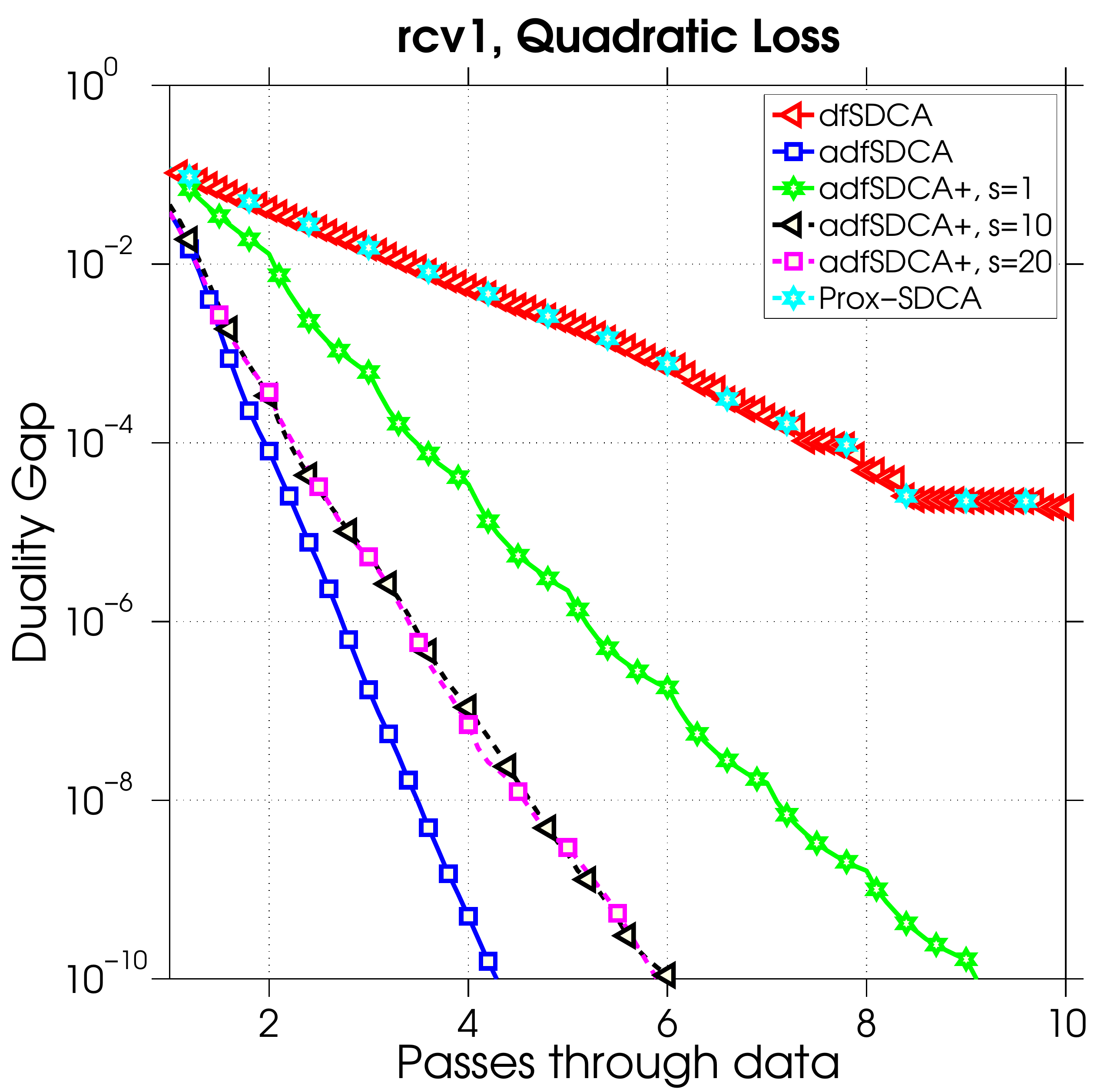}
    \includegraphics[width=0.3\textwidth]{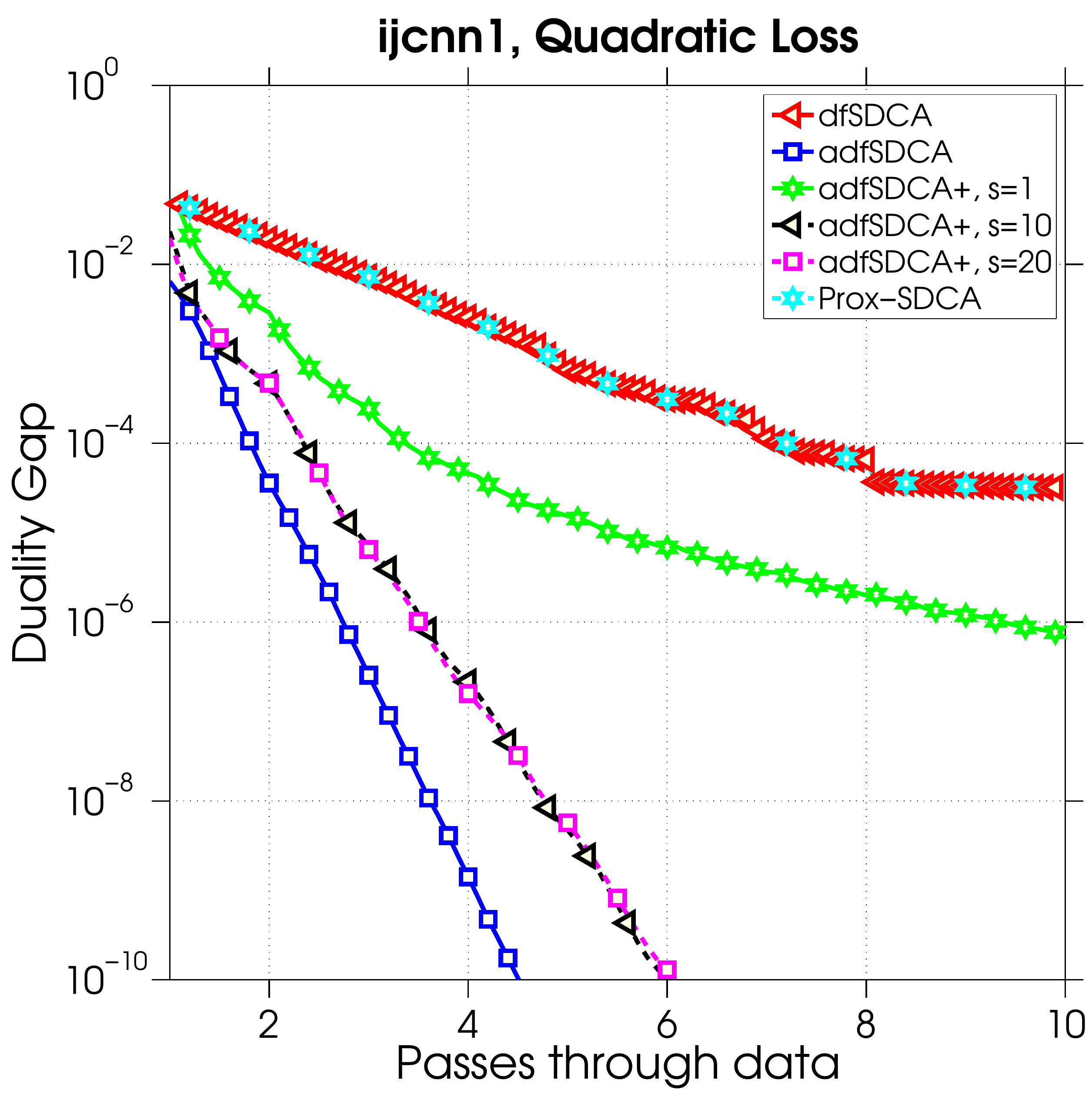}
        \caption{A comparison of the number of epochs versus the duality gap for the various algorithms.}
        \label{Fig:evolution 1}
\end{figure}

\begin{figure}[ht] \centering
        \includegraphics[width=0.3\textwidth]{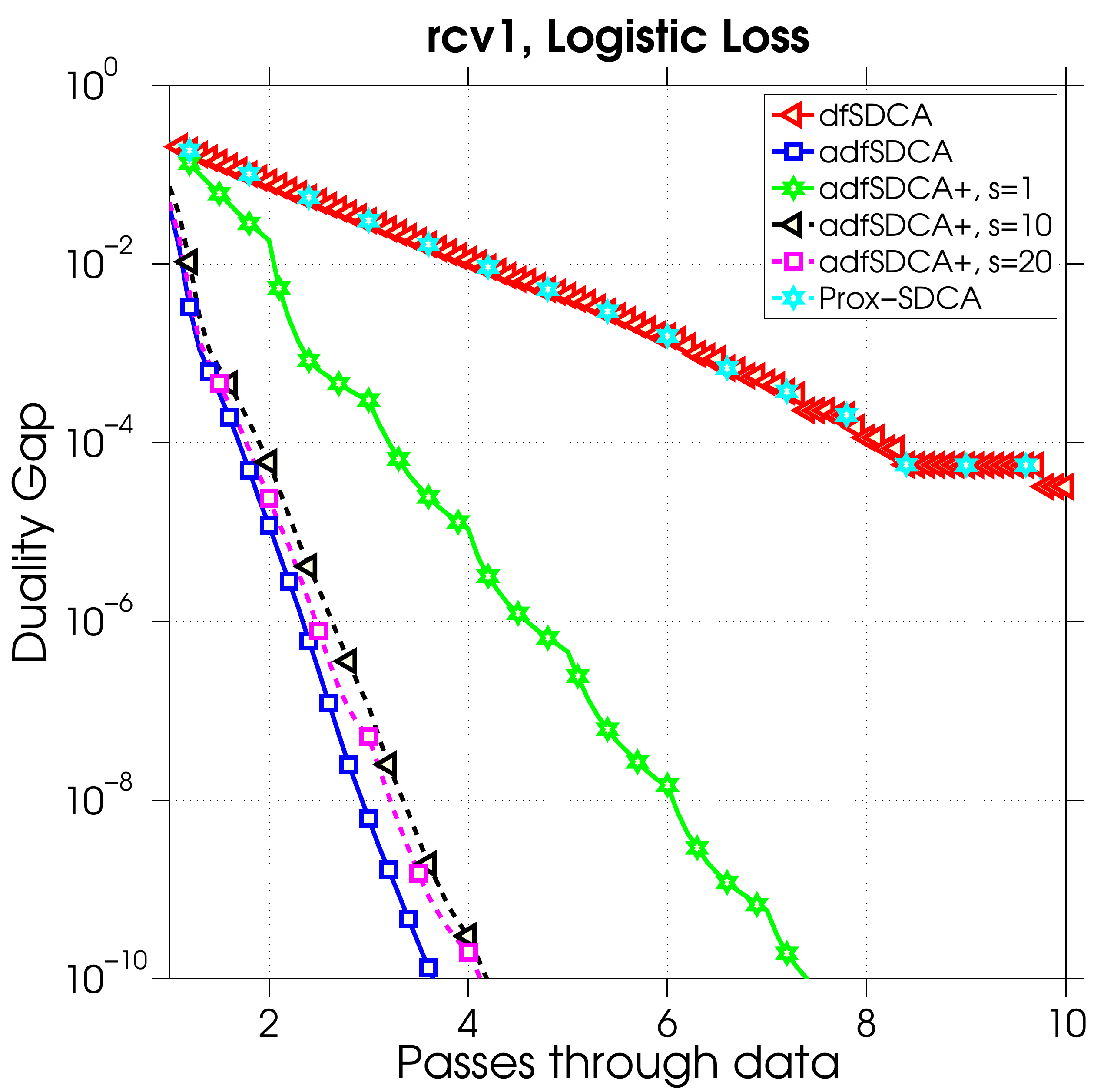}
        \includegraphics[width=0.3\textwidth]{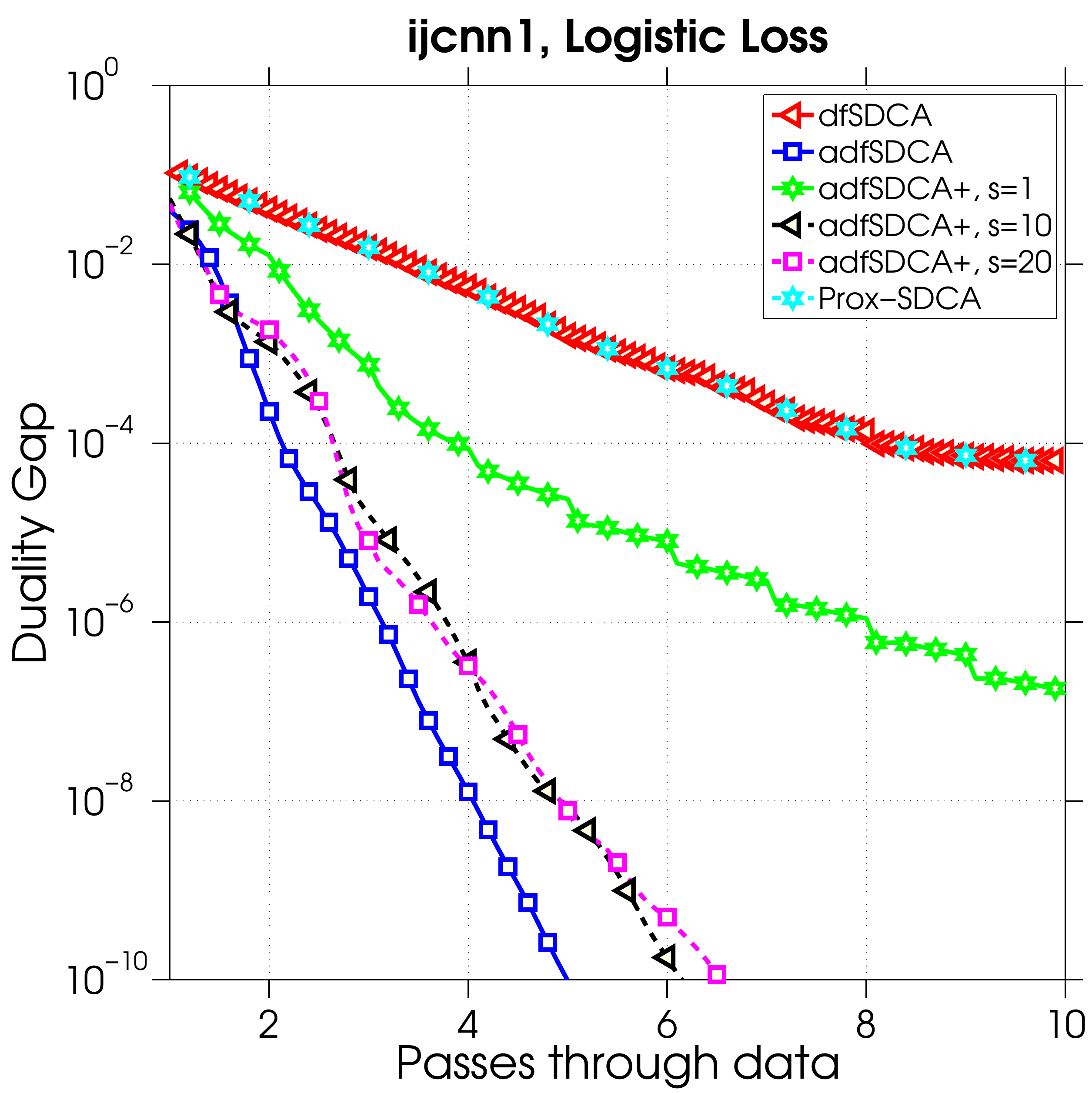}
     \includegraphics[width=0.3\textwidth]{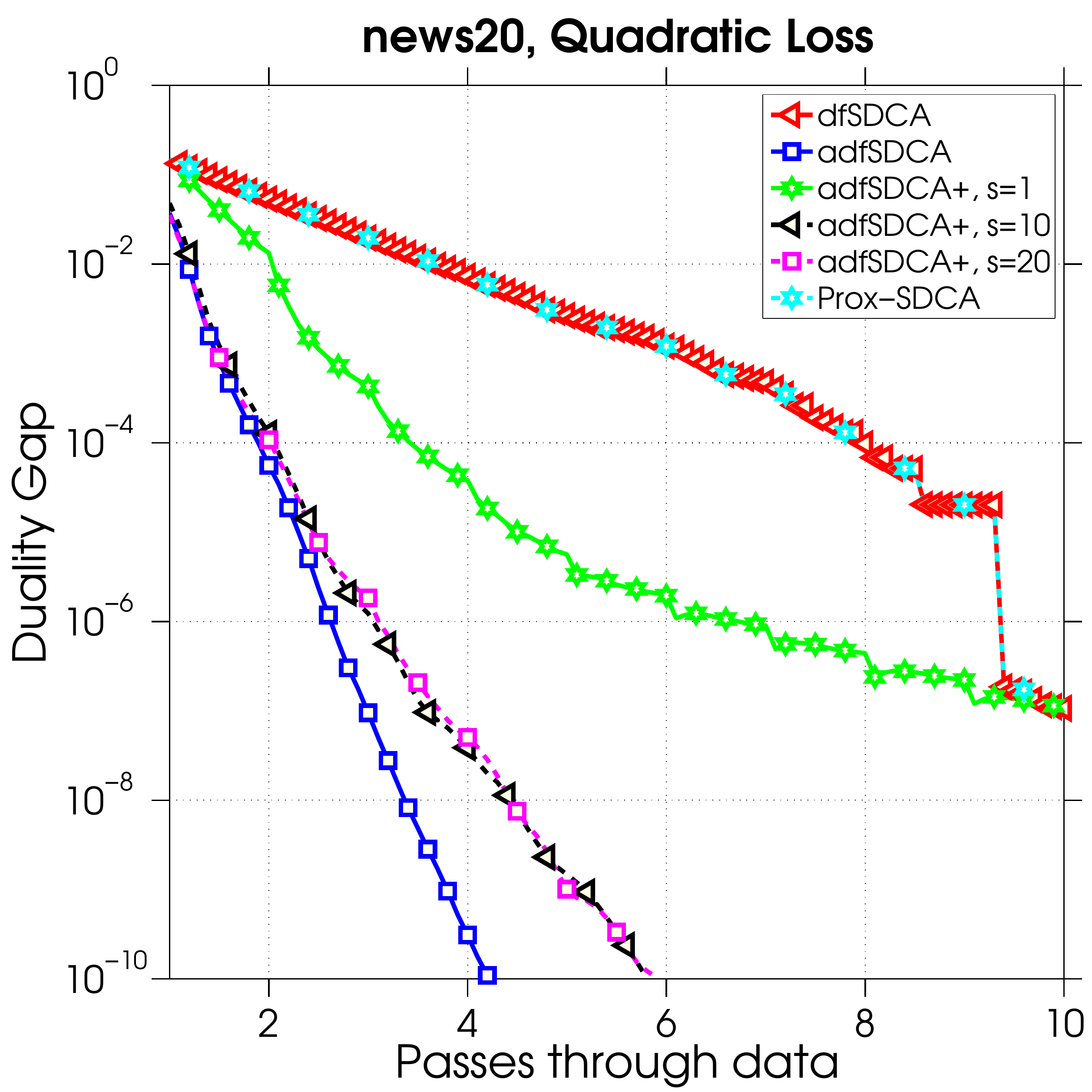}
        \caption{A comparison of the number of epochs versus the duality gap for the various algorithms.}
        \label{Fig:evolution 2}
\end{figure}
Figure \ref{Fig:residuals} shows the estimated density function
of the dual residue $|\vc{\kappa}{t}|$ after $1,2,3,4$ and $5$ epochs for both uniform dfSDCA and our adaptive adfSDCA. One observes that the adaptive scheme is pushing the large residuals towards zero much faster than uniform dfSDCA. For example, notice that after $2$ epochs, almost all residuals are below $0.03$ for adfSDCA, whereas for uniform dfSDCA there are still many residuals larger than $0.06$. This is evidence that, by using adaptive probabilities we are able to update the coordinate with a high dual residue more often and therefore reduce the sub-optimality much more efficiently.
\begin{figure}[ht]\centering
    \includegraphics[width=0.3\textwidth]{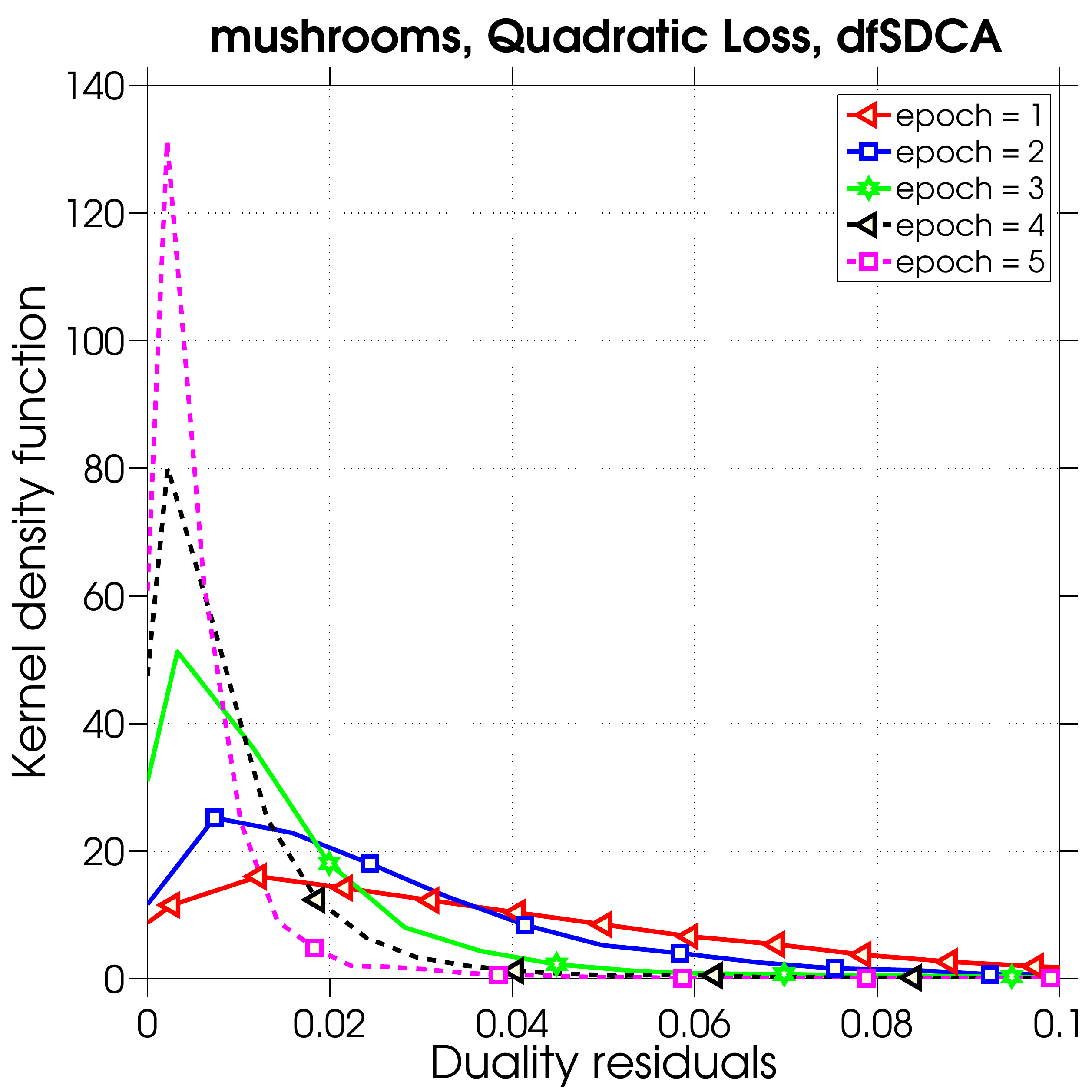}
    \includegraphics[width=0.3\textwidth]{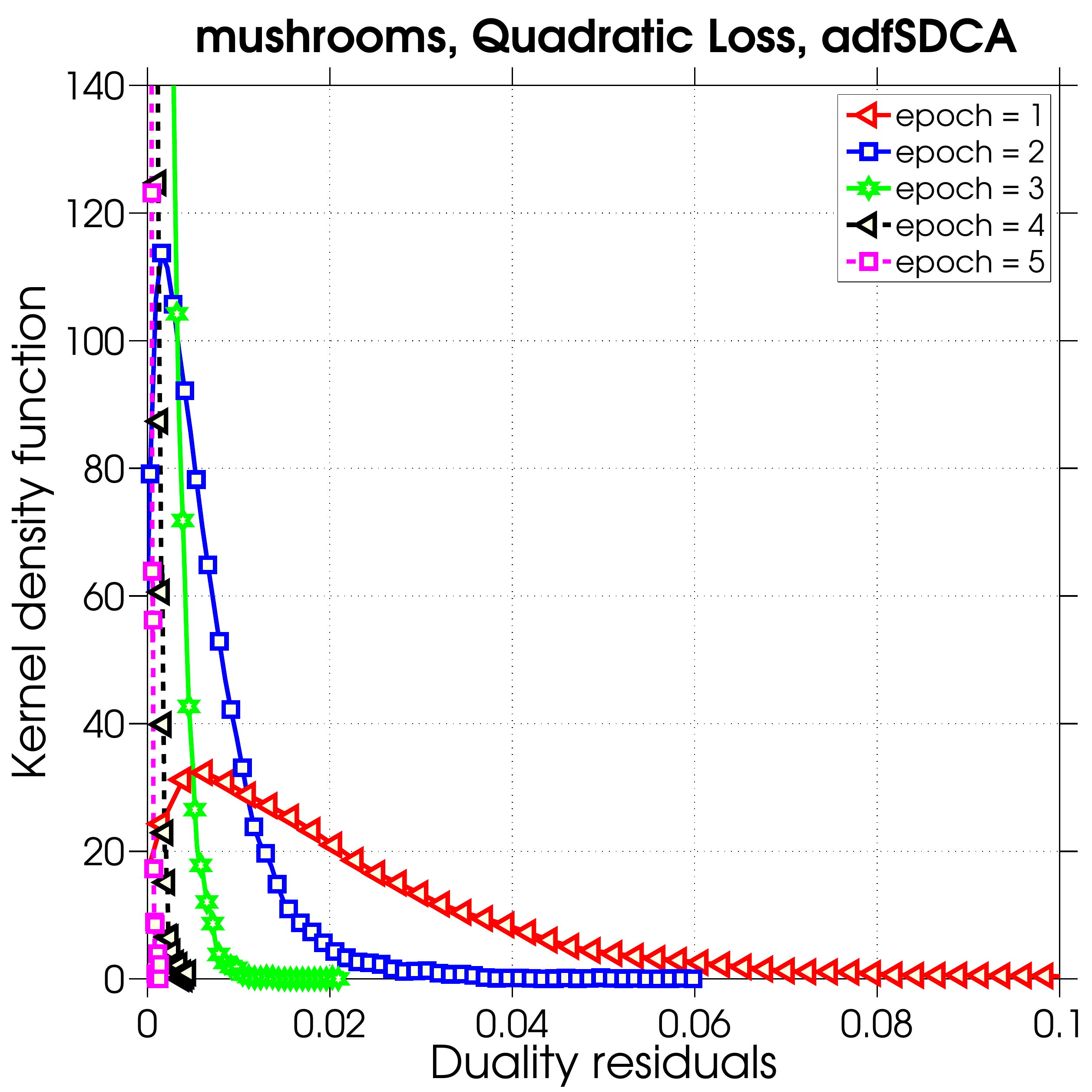}
    \caption{Comparing absolute value of dual residuals at each epoch between dfSDCA and adfSDCA.}
    \label{Fig:residuals}
\end{figure}

\subsection{Mini-batch SDCA}
Here we investigate the behaviour of the mini-batch adfSDCA algorithm (Algorithm \ref{Alg: mini-adfSDCA}). In particular, we compare the practical performance of mini-batch adfSDCA using different mini-batch sizes $b$ varying from $1$ to $32$. Note that if $b=1$, then Algorithm \ref{Alg: mini-adfSDCA} is equivalent to the  adfSDCA algorithm (Algorithm \ref{Alg: adfSDCA}). Figures \ref{exp: minibatch quad 2} and \ref{exp: minibatch quad 3} show that, with respect to the different batch sizes, the mini-batch algorithm with each batch size needs roughly the same number of passes through the data to achieve the same sub-optimality. However, when considering the computational time, the larger the batch size is, the faster the convergence will be. Recall that the results in Section \ref{sec:mini_batch_adfsdca} show that the number of iterations needed by Algorithm \ref{Alg: mini-adfSDCA} used with a batch size of $b$ is roughly $1/b$ times the number of iterations needed by adfSDCA. Here we compute the adaptive probabilities every $b$ samples, which leads to roughly the same number of passes through the data to achieve the same sub-optimality.

\begin{figure}[ht]\centering
     \includegraphics[width=0.3\textwidth]{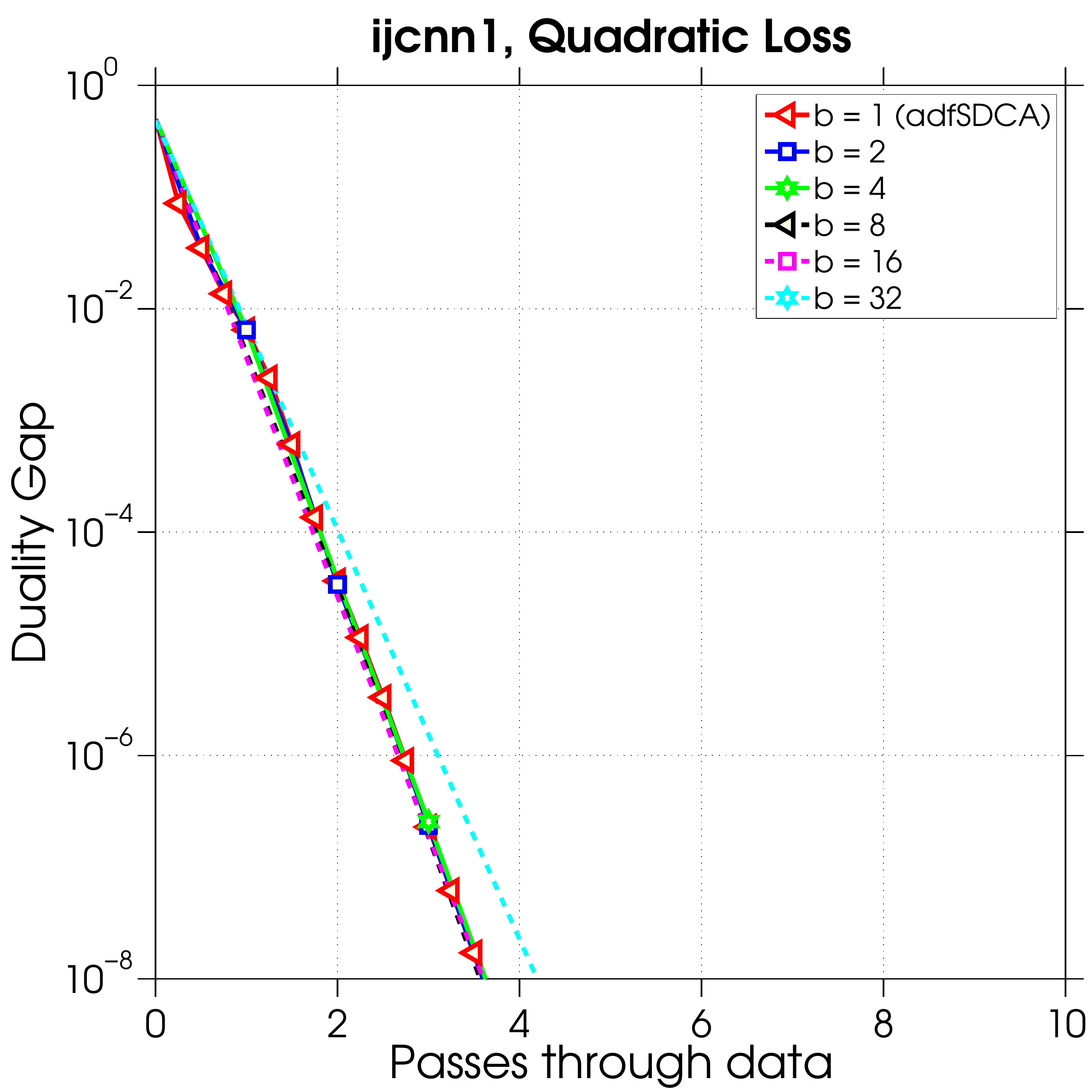}
     \includegraphics[width=0.3\textwidth]{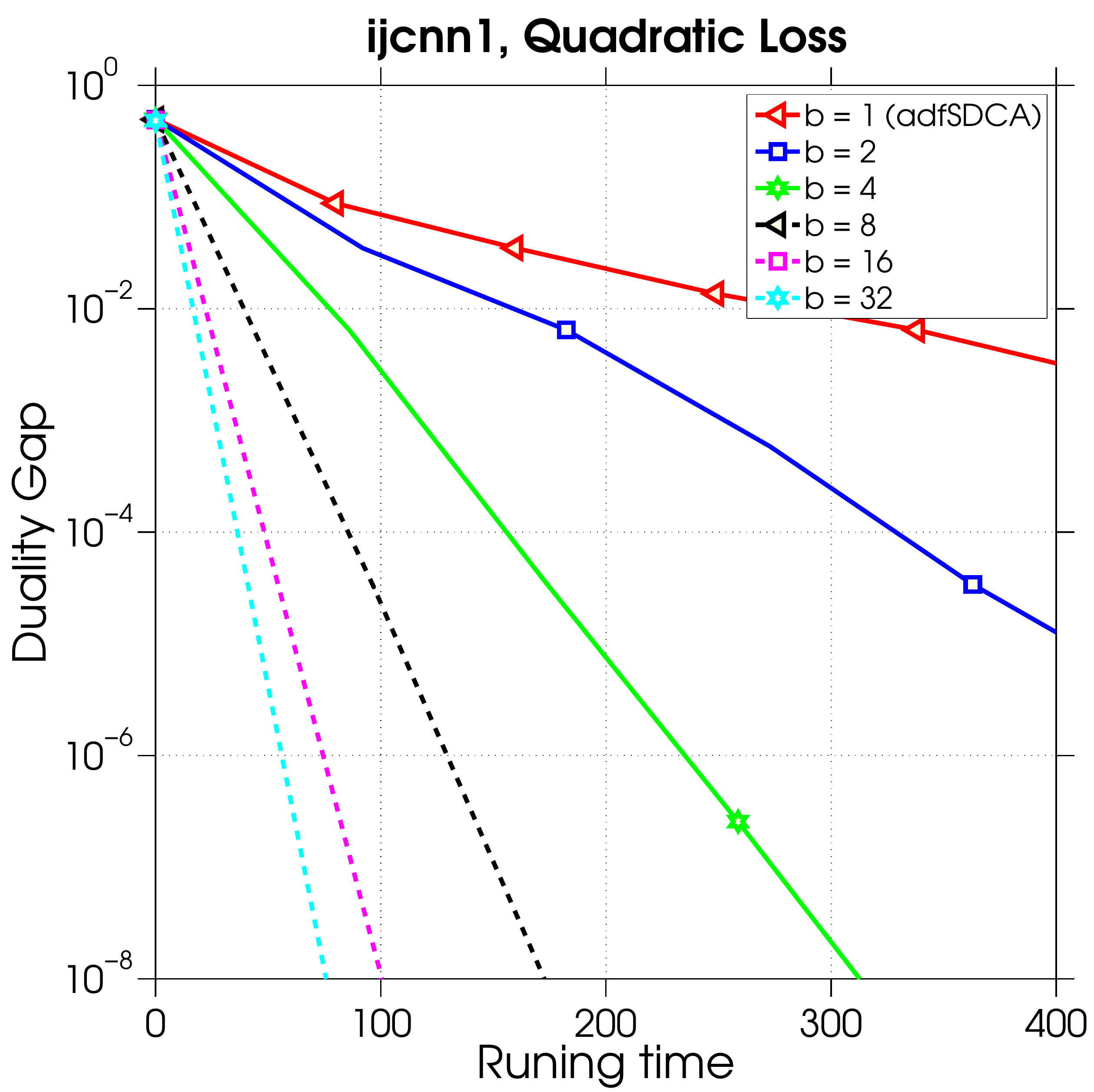}
        \caption{Comparing the number of iterations of various batch size on a quadratic loss.}
        \label{exp: minibatch quad 2}
\end{figure}

\begin{figure}[ht]\centering
    \includegraphics[width=0.3\textwidth]{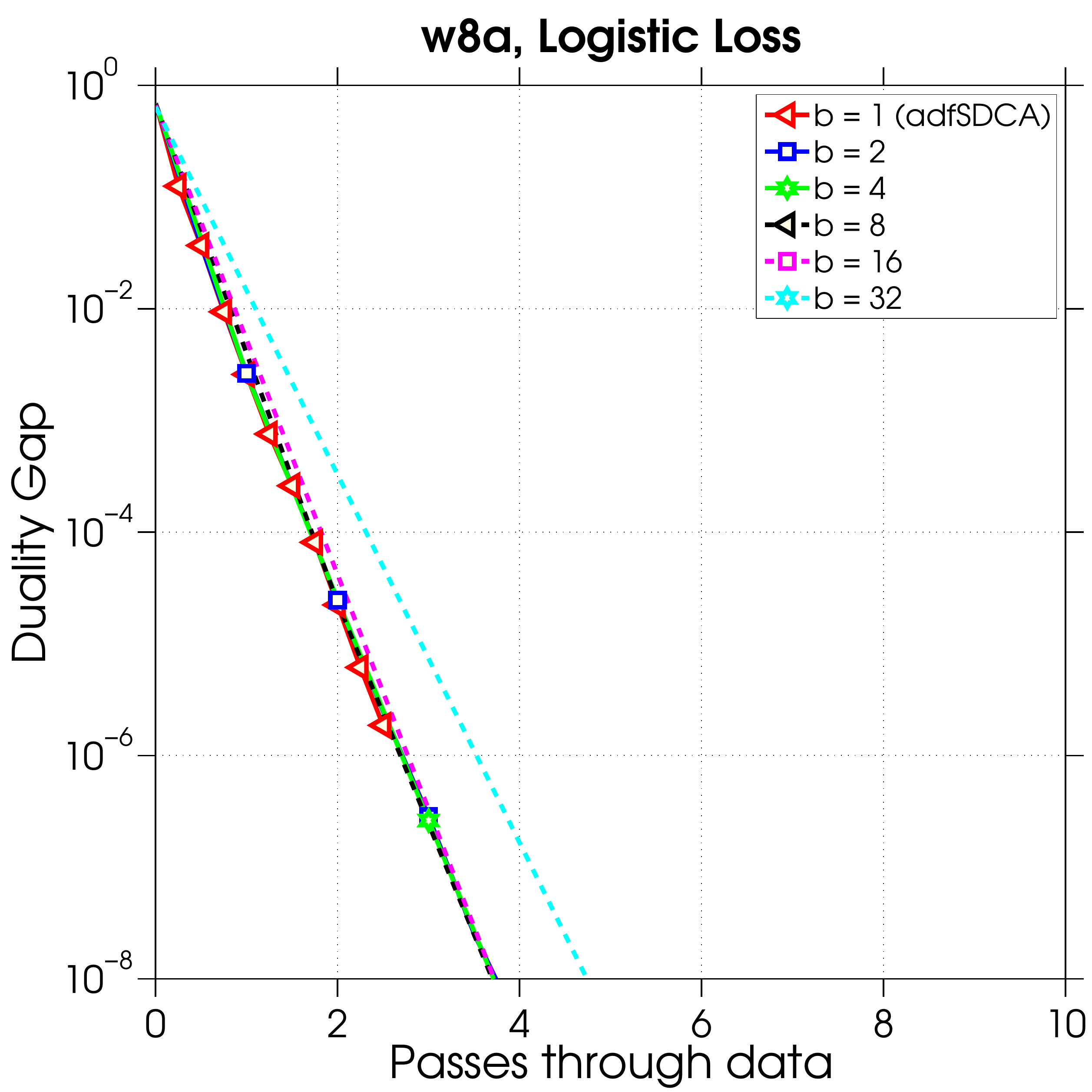}
    \includegraphics[width=0.3\textwidth]{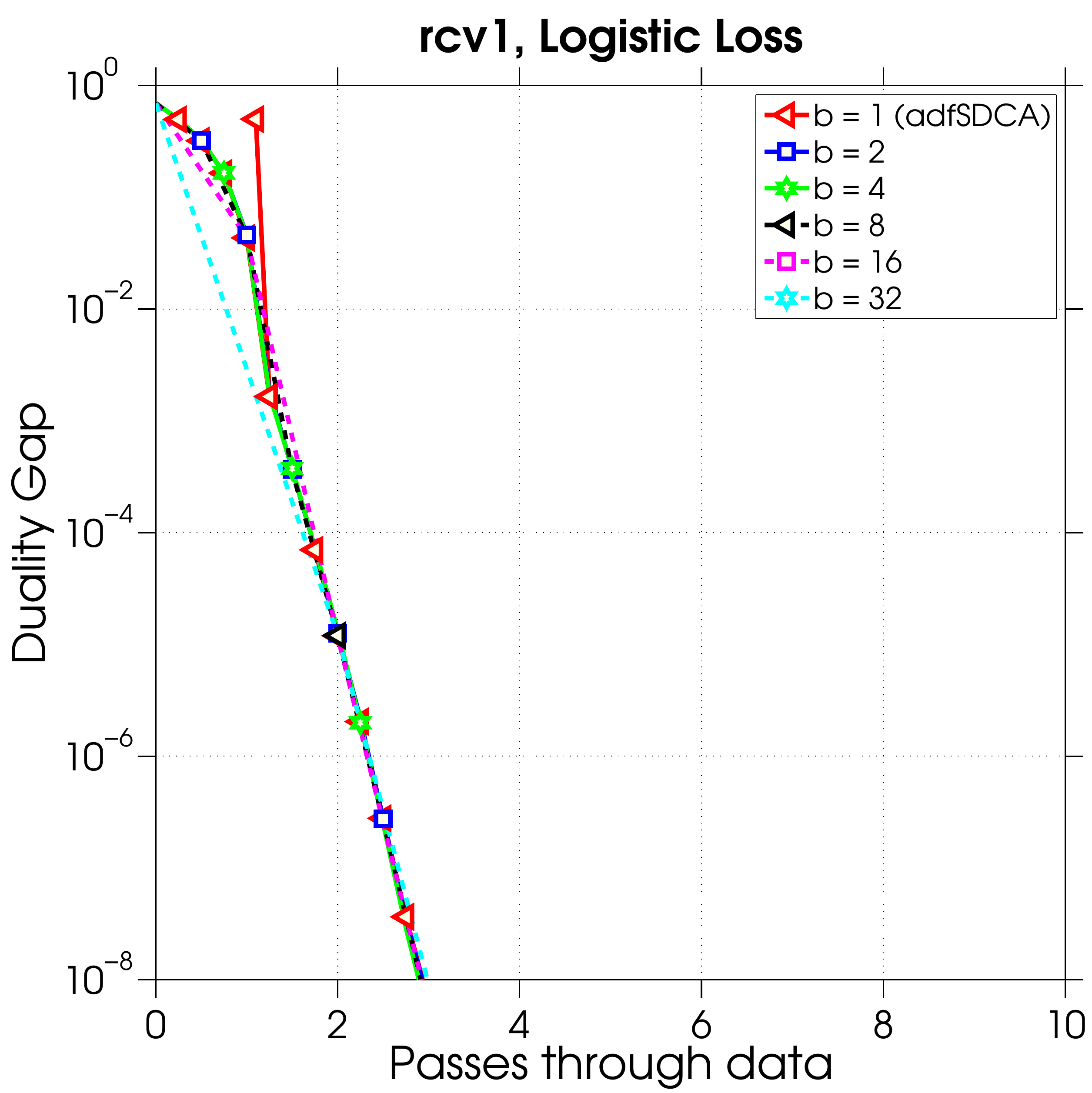}\\
    \includegraphics[width=0.3\textwidth]{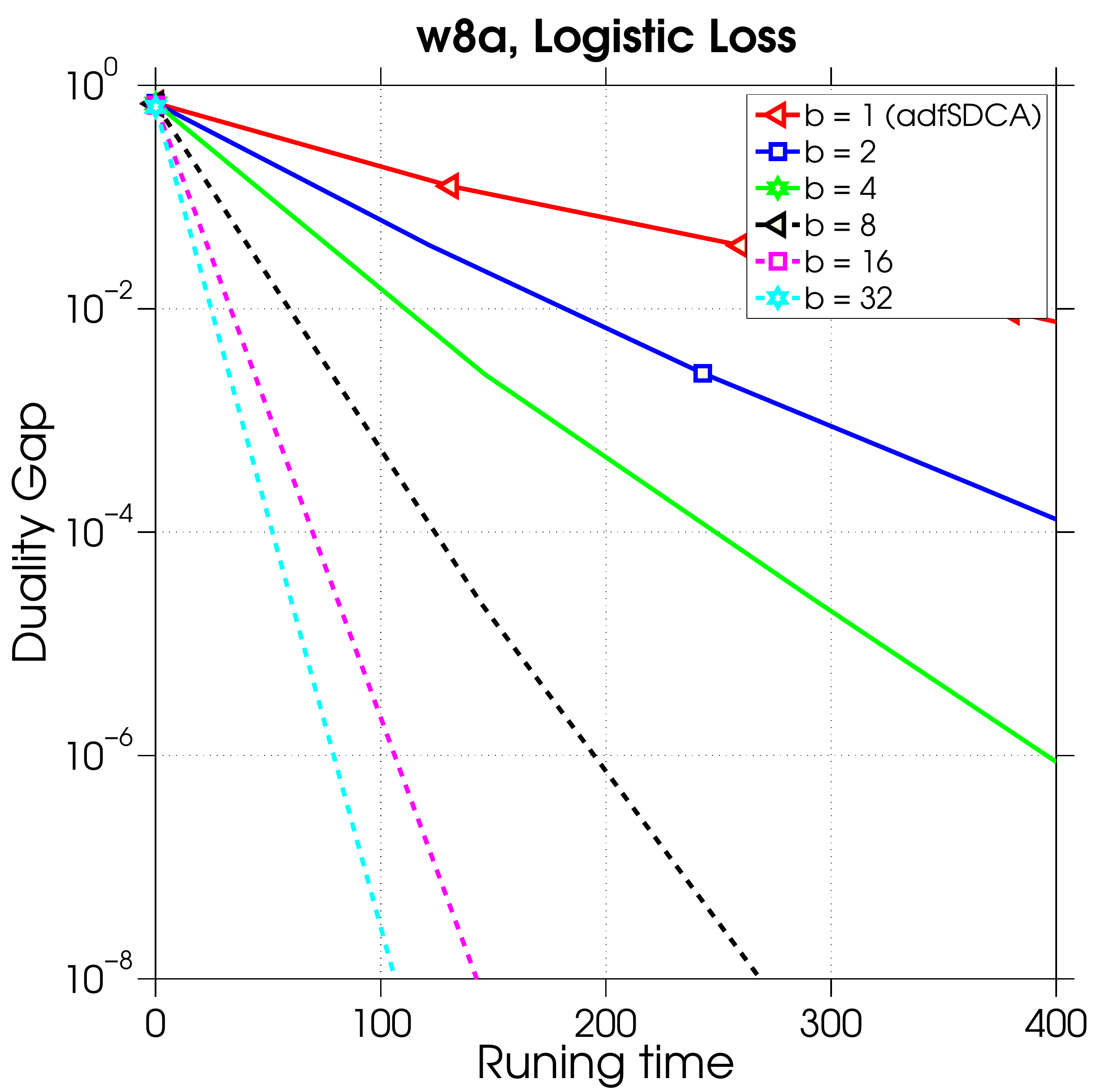}
    \includegraphics[width=0.3\textwidth]{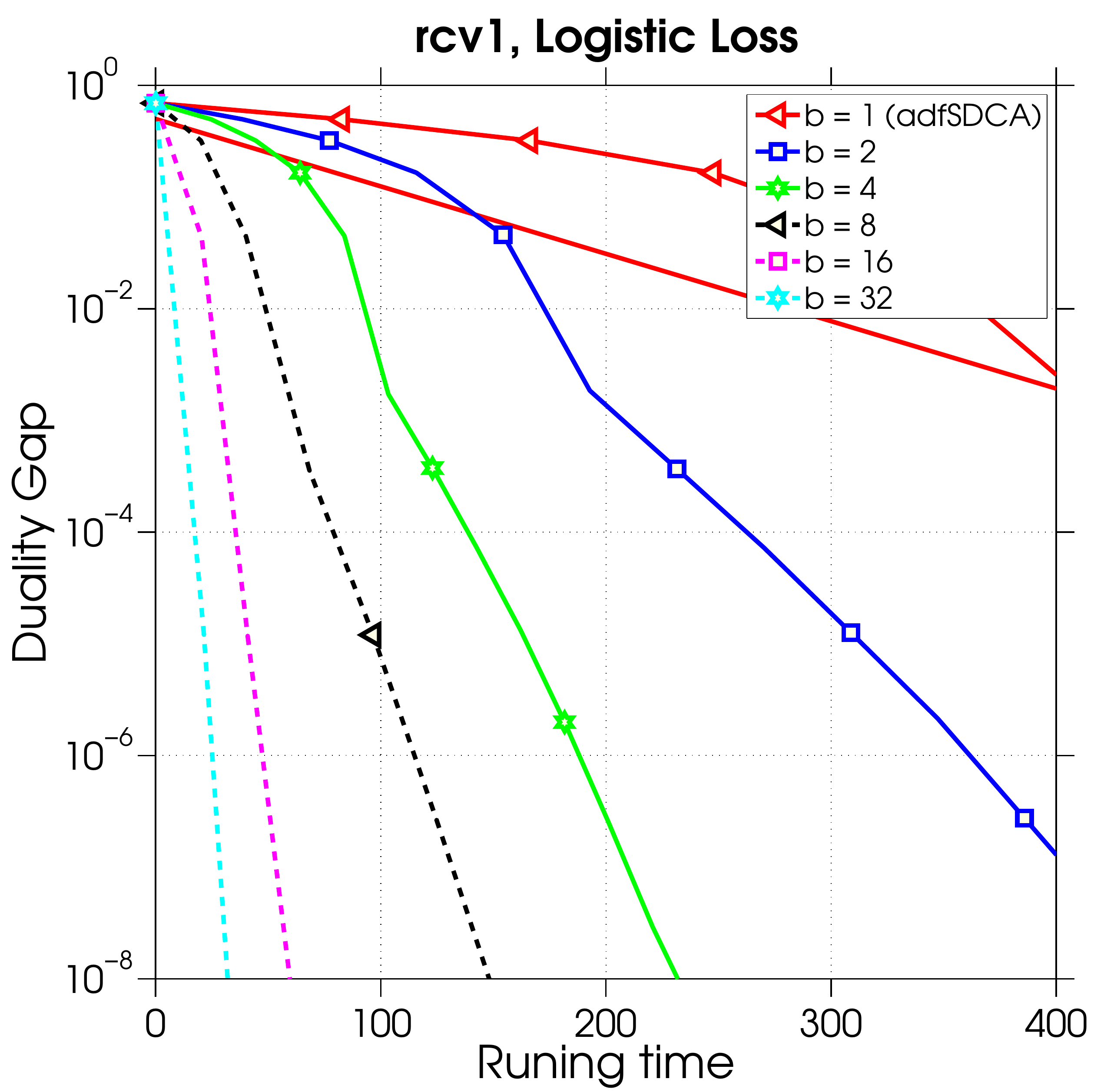}
        \caption{Comparing number of iterations among various batch size on logistic loss.}
        \label{exp: minibatch quad 3}
\end{figure}

\section*{Acknowledgement} We would like to thank Professor Alexander L. Stolyar for his insightful help with Algorithm~\ref{Alg: Non-uniform Mini-batch sampling}. The material is based upon work supported by the U.S. National Science Foundation, under award number NSF:CCF:1618717, NSF:CMMI:1663256 and NSF:CCF:1740796.

\bibliographystyle{frontiersinSCNS_ENG_HUMS}
\bibliography{ref}

\begin{thebibliography}{33}
\providecommand{\natexlab}[1]{#1}
\expandafter\ifx\csname urlstyle\endcsname\relax
  \providecommand{\doi}[1]{doi:\discretionary{}{}{}#1}\else
  \providecommand{\doi}{doi:\discretionary{}{}{}\begingroup
  \urlstyle{rm}\Url}\fi
\providecommand{\selectlanguage}[1]{\relax}
\providecommand{\bibAnnoteFile}[1]{%
  \IfFileExists{#1}{\begin{quotation}\noindent\textsc{Key:} #1\\
  \textsc{Annotation:}\ \input{#1}\end{quotation}}{}}
\providecommand{\bibAnnote}[2]{%
  \begin{quotation}\noindent\textsc{Key:} #1\\
  \textsc{Annotation:}\ #2\end{quotation}}

\bibitem[{Chang and Lin(2011)}]{Chang11}
Chang, C.-C. and Lin, C.-J. (2011).
\newblock Libsvm : a library for support vector machines.
\newblock \emph{ACM Transactions on Intelligent Systems and Technology} 2,
  1--27
\bibAnnoteFile{Chang11}

\bibitem[{Csiba et~al.(2015)Csiba, Qu, and Richt{\'a}rik}]{csiba2015stochastic}
Csiba, D., Qu, Z., and Richt{\'a}rik, P. (2015).
\newblock Stochastic dual coordinate ascent with adaptive probabilities.
\newblock \emph{Proceedings of the 32nd International Conference on Machine
  Learning (ICML-15)} , 674--683
\bibAnnoteFile{csiba2015stochastic}

\bibitem[{Csiba and Richt{\'a}rik(2015)}]{csiba2015primal}
Csiba, D. and Richt{\'a}rik, P. (2015).
\newblock Primal method for {ERM} with flexible mini-batching schemes and
  non-convex losses.
\newblock \emph{arXiv preprint arXiv:1506.02227}
\bibAnnoteFile{csiba2015primal}

\bibitem[{Defazio et~al.(2014)Defazio, Bach, and
  Lacoste-Julien}]{defazio2014saga}
Defazio, A., Bach, F., and Lacoste-Julien, S. (2014).
\newblock Saga: A fast incremental gradient method with support for
  non-strongly convex composite objectives.
\newblock In \emph{Advances in Neural Information Processing Systems}.
  1646--1654
\bibAnnoteFile{defazio2014saga}

\bibitem[{Hsieh et~al.(2008)Hsieh, Chang, Lin, Keerthi, and
  Sundararajan}]{hsieh2008dual}
Hsieh, C.-J., Chang, K.-W., Lin, C.-J., Keerthi, S.~S., and Sundararajan, S.
  (2008).
\newblock A dual coordinate descent method for large-scale linear svm.
\newblock In \emph{Proceedings of the 25th international conference on Machine
  learning} (ACM), 408--415
\bibAnnoteFile{hsieh2008dual}

\bibitem[{Jaggi et~al.(2014)Jaggi, Smith, Tak{\'a}\v{c}, Terhorst, Krishnan,
  Hofmann et~al.}]{jaggi2014communication}
Jaggi, M., Smith, V., Tak{\'a}\v{c}, M., Terhorst, J., Krishnan, S., Hofmann,
  T., et~al. (2014).
\newblock Communication-efficient distributed dual coordinate ascent.
\newblock In \emph{Advances in Neural Information Processing Systems}.
  3068--3076
\bibAnnoteFile{jaggi2014communication}

\bibitem[{Johnson and Zhang(2013)}]{johnson2013accelerating}
Johnson, R. and Zhang, T. (2013).
\newblock Accelerating stochastic gradient descent using predictive variance
  reduction.
\newblock In \emph{Advances in Neural Information Processing Systems}. 315--323
\bibAnnoteFile{johnson2013accelerating}

\bibitem[{Kone{\v{c}}n{\'y} et~al.(2016)Kone{\v{c}}n{\'y}, Liu, Richt{\'a}rik,
  and Tak{\'a}{\v{c}}}]{konevcny2014ms2gd}
Kone{\v{c}}n{\'y}, J., Liu, J., Richt{\'a}rik, P., and Tak{\'a}{\v{c}}, M.
  (2016).
\newblock Mini-batch semi-stochastic gradient descent in the proximal setting.
\newblock \emph{IEEE Journal of Selected Topics in Signal Processing} 10,
  242--255
\bibAnnoteFile{konevcny2014ms2gd}

\bibitem[{Kone{\v{c}}n{\'y} and Richt{\'a}rik(2017)}]{konevcny2013semi}
Kone{\v{c}}n{\'y}, J. and Richt{\'a}rik, P. (2017).
\newblock Semi-stochastic gradient descent methods.
\newblock \emph{Frontiers in Applied Mathematics and Statistics} 3
\bibAnnoteFile{konevcny2013semi}

\bibitem[{Kronmal and Peterson~Jr(1979)}]{kronmal1979alias}
Kronmal, R.~A. and Peterson~Jr, A.~V. (1979).
\newblock On the alias method for generating random variables from a discrete
  distribution.
\newblock \emph{The American Statistician} 33, 214--218
\bibAnnoteFile{kronmal1979alias}

\bibitem[{Liu and Wright(2015)}]{liu2015asynchronous}
Liu, J. and Wright, S.~J. (2015).
\newblock Asynchronous stochastic coordinate descent: Parallelism and
  convergence properties.
\newblock \emph{SIAM Journal on Optimization} 25, 351--376
\bibAnnoteFile{liu2015asynchronous}

\bibitem[{Ma et~al.(2015)Ma, Smith, Jaggi, Jordan, Richt{\'a}rik, and
  Tak{\'a}{\v{c}}}]{ma2015adding}
Ma, C., Smith, V., Jaggi, M., Jordan, M.~I., Richt{\'a}rik, P., and
  Tak{\'a}{\v{c}}, M. (2015).
\newblock Adding vs. averaging in distributed primal-dual optimization.
\newblock \emph{In 32th International Conference on Machine Learning, ICML
  2015}
\bibAnnoteFile{ma2015adding}

\bibitem[{Necoara and Clipici(2013)}]{necoara2013efficient}
Necoara, I. and Clipici, D. (2013).
\newblock Efficient parallel coordinate descent algorithm for convex
  optimization problems with separable constraints: application to distributed
  mpc.
\newblock \emph{Journal of Process Control} 23, 243--253
\bibAnnoteFile{necoara2013efficient}

\bibitem[{Necoara and Clipici(2016)}]{necoara2013parallel}
Necoara, I. and Clipici, D. (2016).
\newblock Parallel random coordinate descent method for composite minimization.
\newblock \emph{SIAM Journal on Optimization} 26, 197--226
\bibAnnoteFile{necoara2013parallel}

\bibitem[{Nesterov(2012)}]{nesterov2012efficiency}
Nesterov, Y. (2012).
\newblock Efficiency of coordinate descent methods on huge-scale optimization
  problems.
\newblock \emph{SIAM Journal on Optimization} 22, 341--362
\bibAnnoteFile{nesterov2012efficiency}

\bibitem[{Nitanda(2014)}]{nitanda2014stochastic}
Nitanda, A. (2014).
\newblock Stochastic proximal gradient descent with acceleration techniques.
\newblock In \emph{Advances in Neural Information Processing Systems}.
  1574--1582
\bibAnnoteFile{nitanda2014stochastic}

\bibitem[{Qu and Richt{\'a}rik(2016)}]{qu2014coordinate}
Qu, Z. and Richt{\'a}rik, P. (2016).
\newblock Coordinate descent with arbitrary sampling {II}: Expected separable
  overapproximation.
\newblock \emph{Optimization Methods and Software} 31, 858--884
\bibAnnoteFile{qu2014coordinate}

\bibitem[{Qu et~al.(2015)Qu, Richt{\'a}rik, and Zhang}]{qu2015quartz}
Qu, Z., Richt{\'a}rik, P., and Zhang, T. (2015).
\newblock Quartz: Randomized dual coordinate ascent with arbitrary sampling.
\newblock In \emph{Advances in Neural Information Processing Systems}. 865--873
\bibAnnoteFile{qu2015quartz}

\bibitem[{Richt{\'a}rik and Tak{\'a}{\v{c}}(2012)}]{richtarik2012parallel}
Richt{\'a}rik, P. and Tak{\'a}{\v{c}}, M. (2012).
\newblock Parallel coordinate descent methods for big data optimization.
\newblock \emph{Mathematical Programming} , 1--52
\bibAnnoteFile{richtarik2012parallel}

\bibitem[{Richt{\'a}rik and Tak{\'a}{\v{c}}(2014)}]{richtarik2014iteration}
Richt{\'a}rik, P. and Tak{\'a}{\v{c}}, M. (2014).
\newblock Iteration complexity of randomized block-coordinate descent methods
  for minimizing a composite function.
\newblock \emph{Mathematical Programming} 144, 1--38
\bibAnnoteFile{richtarik2014iteration}

\bibitem[{Roux et~al.(2012)Roux, Schmidt, and Bach}]{roux2012stochastic}
Roux, N.~L., Schmidt, M., and Bach, F.~R. (2012).
\newblock A stochastic gradient method with an exponential convergence rate for
  finite training sets.
\newblock In \emph{Advances in Neural Information Processing Systems}.
  2663--2671
\bibAnnoteFile{roux2012stochastic}

\bibitem[{Schmidt et~al.(2017)Schmidt, Roux, and Bach}]{schmidt2013minimizing}
Schmidt, M., Roux, N.~L., and Bach, F. (2017).
\newblock Minimizing finite sums with the stochastic average gradient.
\newblock \emph{Mathematical Programming} 162, 83--112
\bibAnnoteFile{schmidt2013minimizing}

\bibitem[{Shalev-Shwartz(2015)}]{DBLP:journals/corr/Shalev-Shwartz15}
Shalev-Shwartz, S. (2015).
\newblock {SDCA} without duality.
\newblock \emph{arXiv preprint arXiv:1502.06177}
\bibAnnoteFile{DBLP:journals/corr/Shalev-Shwartz15}

\bibitem[{Shalev-Shwartz(2016)}]{shalev2016sdca}
Shalev-Shwartz, S. (2016).
\newblock Sdca without duality, regularization, and individual convexity.
\newblock 747--754
\bibAnnoteFile{shalev2016sdca}

\bibitem[{Shalev-Shwartz and Ben-David(2014)}]{Shalev14}
Shalev-Shwartz, S. and Ben-David, S. (2014).
\newblock \emph{Understanding Machine Learning: From Theory to Algorithms} (New
  York, USA: Cambridge University Press)
\bibAnnoteFile{Shalev14}

\bibitem[{Shalev-Shwartz et~al.(2011)Shalev-Shwartz, Singer, Srebro, and
  Cotter}]{shalev2011pegasos}
Shalev-Shwartz, S., Singer, Y., Srebro, N., and Cotter, A. (2011).
\newblock Pegasos: Primal estimated sub-gradient solver for svm.
\newblock \emph{Mathematical programming} 127, 3--30
\bibAnnoteFile{shalev2011pegasos}

\bibitem[{Shalev-Shwartz and Zhang(2013)}]{shalev2013stochastic}
Shalev-Shwartz, S. and Zhang, T. (2013).
\newblock Stochastic dual coordinate ascent methods for regularized loss.
\newblock \emph{The Journal of Machine Learning Research} 14, 567--599
\bibAnnoteFile{shalev2013stochastic}

\bibitem[{Shalev-Shwartz and Zhang(2014)}]{shalev2014accelerated}
Shalev-Shwartz, S. and Zhang, T. (2014).
\newblock Accelerated proximal stochastic dual coordinate ascent for
  regularized loss minimization.
\newblock \emph{Mathematical Programming} , 1--41
\bibAnnoteFile{shalev2014accelerated}

\bibitem[{Tak{\'a}{\v{c}} et~al.(2013)Tak{\'a}{\v{c}}, Bijral, Richt{\'a}rik,
  and Srebro}]{takavc2013mini}
Tak{\'a}{\v{c}}, M., Bijral, A., Richt{\'a}rik, P., and Srebro, N. (2013).
\newblock Mini-batch primal and dual methods for svms.
\newblock \emph{Proceedings of the 30th International Conference on Machine
  Learning}
\bibAnnoteFile{takavc2013mini}

\bibitem[{Tak{\'a}{\v{c}} et~al.(2015)Tak{\'a}{\v{c}}, Richt{\'a}rik, and
  Srebro}]{takavc2015distributed}
Tak{\'a}{\v{c}}, M., Richt{\'a}rik, P., and Srebro, N. (2015).
\newblock Distributed mini-batch {SDCA}.
\newblock \emph{arXiv preprint arXiv:1507.08322}
\bibAnnoteFile{takavc2015distributed}

\bibitem[{Tappenden et~al.(2017)Tappenden, Tak{\'a}{\v{c}}, and
  Richt{\'a}rik}]{tappenden2015complexity}
Tappenden, R., Tak{\'a}{\v{c}}, M., and Richt{\'a}rik, P. (2017).
\newblock On the complexity of parallel coordinate descent.
\newblock \emph{Optimization Methods and Software} , 1--24
\bibAnnoteFile{tappenden2015complexity}

\bibitem[{Zhang and Xiao(2015)}]{zhang2015communication}
Zhang, Y. and Xiao, L. (2015).
\newblock Disco: distributed optimization for self-concordant empirical loss.
\newblock \emph{Proceedings of the 32nd International Conference on
  International Conference on Machine Learning (ICML-15)} , 362--370
\bibAnnoteFile{zhang2015communication}

\bibitem[{Zhao and Zhang(2015)}]{zhao2014stochastic}
Zhao, P. and Zhang, T. (2015).
\newblock Stochastic optimization with importance sampling for regularized loss
  minimization.
\newblock \emph{Proceedings of the 32nd International Conference on Machine
  Learning (ICML-15)} , 1--9
\bibAnnoteFile{zhao2014stochastic}

\end{thebibliography}

\clearpage

\appendix
\onecolumn

\renewcommand\thesection{\Alph{section}}
\renewcommand\thesubsection{\thesection.\arabic{subsection}}
\section{Appendix}
\subsection{Preliminaries and Technical Results}
\label{sec:appendix}
Recall that $w^*$ denotes an optimum of \eqref{Prob: L2EMR} and define $\alpha_i^* = -\phi'_i(x_i^Tw^*)$. To simplify the proofs we introduce the following variables
\begin{align}\label{eq:AB}
    A^{(t)} = \tfrac{1}{n}\norm{\alpha^{(t)}-\alpha^*}^2 \m{and} B^{(t)} = \norm{w^{(t)}-w^*}^2.
\end{align}
At the optimum $w^*$, it holds that $0 = \nabla P(w^*) = \frac{1}{n}\sum_{i=1}^n\phi'_i(x_i^Tw^*)x_i + \lambda w^*$, so $w^* = \frac{1}{\lambda n}\sum_{i=1}^n\alpha_i^*x_i$. Define $u_i^{(t)} \eqdef -\phi'_i(x_i^Tw^{(t)})$, and therefore we have $\kappa_i^{(t)} = \alpha_i^{(t)} - u_i^{(t)}$ and $u_i^* = \alpha_i^*$.

The following two lemmas will be useful when proving our main results.
\begin{lemma}
Let $A\ttt$ and $B\ttt$ be defined in \eqref{eq:AB}, and let $v_i = \norm{x_i}^2$ for all $i\in[n]$. Then, conditioning on $\alpha\ttt$, the following hold for given $\theta$:
     \begin{align}
        \E[A^{(t+1)}|\alpha\ttt] - A^{(t)} &= - \theta A^{(t)}+ \frac{\theta}{n}\sum_{i=1}^n\Big((u_i^{(t)}-\alpha_i^*)^2 - \big(1- \tfrac{\theta}{p_i}\big)(\kappa_i^{(t)})^2\Big),\label{eq: E[A]}\\
        \E[B^{(t+1)}|\alpha\ttt] - B^{(t)}  & = -\tfrac{2 \theta }{\lambda}(w^{(t)}-w^*)^T \nabla P(w^{(t)}) + \sum_{i=1}^n \frac{\theta^2v_i}{n^2 \lambda^2p_i}\big(\kappa^{(t)}_i\big)^2.\label{eq: E[B]}
    \end{align}
\end{lemma}
\begin{proof}

    Note that at iteration $t$, only coordinate $i$ (of $\alpha$) is updated, so
    \begin{eqnarray}\label{eq:intA}
      A^{(t+1)} =\tfrac{1}{n}\norm{\alpha^{(t+1)}-\alpha^*}^2 \equiv \frac1n \sum_{j\neq i} (\alpha_j\ttt- \alpha_j^*)^2  + \frac1n (\alpha_i^{(t+1)} - \alpha_i^*)^2.
    \end{eqnarray}
    Thus,
    \begin{align}
        A^{(t+1)} - A^{(t)}
        \overset{\eqref{eq:intA}}{=}&\tfrac{1}{n}(\alpha_i^{(t+1)}-\alpha_i^*)^2 - \tfrac{1}{n}(\alpha_i^{(t)}-\alpha_i^*)^2\notag\\
        =&\tfrac{1}{n}\Big(\alpha_i^{(t)} - \tfrac{\theta}{p_i}\kappa_i^{(t)}-\alpha_i^*\Big)^2 - \tfrac{1}{n}(\alpha_i^{(t)} - \alpha_i^*)^2\notag\\
        =&\tfrac{1}{n}\left(\big(1-\tfrac{\theta}{p_i}\big)(\alpha_i^{(t)}-\alpha_i^*) - \tfrac{\theta}{p_i}(\kappa_i^{(t)}-\alpha_i^*)\right)^2 - \tfrac{1}{n}(\alpha_i^{(t)} - \alpha_i^*)^2\notag\\
        =&\tfrac{1}{n}\left(1- \tfrac{\theta}{p_i}\right)(\alpha_i^{(t)}-\alpha_i^*)^2+\tfrac{\theta}{np_i}(u_i^{(t)}-\alpha_i^*)^2-\left(1- \tfrac{\theta}{p_i}\right)\tfrac{\theta}{np_i}(\kappa_i^{(t)})^2 - \tfrac{1}{n}(\alpha_i^{(t)} - \alpha_i^*)^2\notag\\
        =&-\tfrac{\theta}{np_i}(\alpha_i^{(t)}-\alpha_i^*)^2+\tfrac{\theta}{np_i}\left((u_i^{(t)}-\alpha_i^*)^2 - \left(1- \tfrac{\theta}{p_i}\right)(\kappa_i^{(t)})^2\right).
    \end{align}
    Taking expectation over $i\in[n]$, conditioned on $\alpha^{(t)}$, gives the first result.

    To obtain the second result consider
    \begin{eqnarray*}
        B^{(t+1)} - B^{(t)} &\overset{\eqref{eq:AB}}{=}& \norm{w^{(t+1)}-w^*}^2 - \norm{w^{(t)}-w^*}^2\\
        &\overset{\eqref{wupdate}}{=}&\norm{w^{(t)} - \tfrac{\theta}{n \lambda p_i}\kappa_i^{(t)}x_i-w^*}^2 - \norm{w^{(t)}-w^*}^2\\
        &=& -\tfrac{2 \theta}{n \lambda p_i}\kappa_i^{(t)}x_i^T(w^{(t)}-w^*)+ \tfrac{\theta^2v_i}{n^2 \lambda^2 p_i^2}(\kappa_i^{(t)})^2.
    \end{eqnarray*}
     Recall that $\E[\frac{1}{np_i}\kappa_i^{(t)}x_i] = \nabla P(w^{(t)})$ by \eqref{eq: nabla(P)}. Thus, taking expectation over $i\in [n]$, conditioned on $w^{(t)}$, gives \eqref{eq: E[B]}.\qed
\end{proof}

The following Lemma and proof are similar to \cite[Lemma~4]{csiba2015primal} and \cite[Lemma~1]{shalev2013stochastic}.
\begin{lemma}
    Assume that each $\phi_i$ is $\tilde{L}_i$-smooth and convex. Then, for every $w$
    \begin{align}\label{eq: strong ineq}
     \frac{1}{\tilde{L}}\Big(\frac{1}{n}\sum_{i=1}^n\norm{\phi'_i(x_i^Tw) - \phi'_i(x_i^Tw^*)}^2\Big)
     &\leq\frac{1}{n}\sum_{i=1}^n \frac{1}{\tilde{L}_i}\norm{\phi'_i(w^Tx_i) - \phi'_i(x_i^Tw^*)}^2\nonumber\\
     &\leq 2\left(P(w)-P(w^*)- \tfrac{\lambda}{2}\norm{w-w^*}^2\right).
\end{align}
\end{lemma}

\begin{proof}
    Let $z, z^* \in \R$. Define
    \begin{equation}\label{eq:gidef}
        g_i(z) \eqdef \phi_i(z) - \phi_i(z^*) - \phi'_i(z^*)(z-z^*).
    \end{equation}
Because $\phi_i$ is $\tilde{L}_i$-smooth, so too is $g_i$, which implies that for all $z, \hat z \in \R$,
   \begin{equation}\label{giLipsch}
       g_i(z)  \leq  g_i(\hat{z}) + g'_i(\hat{z})(z-\hat{z}) + \tfrac{\tilde{L}_i}{2}(z-\hat{z})^2.
   \end{equation}
By convexity of $\phi_i$, $g_i$ is nonnegative, i.e., $g_i(z)\geq 0$ for all $z$. Hence, by non-negativity and smoothness $g_i$ is self-bounded (see Section 12.1.3 in \cite{Shalev14} or set $z = \hat{z} - \frac1{\tilde L_i}g'_i(\hat z)$ in \eqref{giLipsch} and rearrange):
    \begin{equation}\label{eq:gibounded}
       \norm{g'_i(z)}^2 \leq 2 \tilde{L}_i g_i(z), \qquad \forall z.
   \end{equation}
Differentiating \eqref{eq:gidef} w.r.t. $z$ and combining the result with \eqref{eq:gibounded}, used with $z = x_i^Tw$ and $z^* = x_i^Tw^*$, gives
   \begin{equation}\label{eq:gieq}
       \norm{\phi'_i(x_i^Tw) - \phi'_i(x_i^Tw^*)}^2 = \norm{g'_i(x_i^Tw)}^2 \leq 2\tilde{L}_ig_i(x_i^Tw).
   \end{equation}
Multiplying \eqref{eq:gieq} through by $1/(n\tilde L_i)$ and summing over $i\in [n]$ shows that
   \begin{align*}
       &\frac{1}{n}\sum_{i=1}^n \frac{1}{\tilde{L}_i}\norm{\phi'_i(x_i^Tw) - \phi'_i(x_i^Tw^*)}^2 \leq \frac{2}{n}\sum_{i=1}^ng_i(x_i^Tw)\\
       = ~&\frac{2}{n}\sum_{i=1}^n\phi_i(x_i^Tw) - \phi_i(x_i^Tw^*) - \phi'_i(x_i^Tw^*)(x_i^Tw-x_i^Tw^*)\\
       = ~& 2\left(P(w)-\tfrac{\lambda}{2}\norm{w}^2-P(w^*) + \tfrac{\lambda}{2}\norm{w^*}^2 - \lambda (w^*)^T(w-w^*)\right)\\
       = ~&2\left(P(w)-P(w^*)- \tfrac{\lambda}{2}\norm{w-w^*}^2\right),
   \end{align*}
 where we have used the fact that $\E[\nabla P(w^*)] = \E[\phi'(x_i^Tw^*) x_i + \lambda w^*] = 0.$
The first inequality follows because $\tilde{L} = \max_i \tilde{L}_i$.\qed
\end{proof}

\subsection{Proof of Lemmas \ref{thm: seperate convex} and \ref{thm: average convex}}

\begin{proof}[Proof of Lemma \ref{thm: seperate convex}]
In this case it is assumed that every loss function is convex and we set $\gamma= \lambda \tilde{L}$ \eqref{eq:gamma}.
For convenience, define the following quantities:
\begin{equation}\label{eq:C1}
  \mathbf{C_1} \eqdef \frac{\theta}{n}\sum_{i=1}^n(u_i^{(t)}-\alpha_i^*)^2-\tfrac{2 \gamma\theta}{\lambda}\nabla P(w^{(t)})^T(w^{(t)}-w^*)
\end{equation}
\begin{equation}\label{eq:C2}
  \mathbf{C_2} \eqdef \sum_{i=1}^n\Big(-\tfrac{\theta}{n}(1- \tfrac{\theta}{p_i})+\tfrac{\theta^2v_i \gamma}{n^2 \lambda^2 p_i}\Big)(\kappa_i^{(t)})^2
\end{equation}
Recall that $A\ttt$, $B\ttt$ and $D\ttt$ are defined in \eqref{eq:AB} and \eqref{eq:D}, respectively, and $\gamma$ is defined in \eqref{eq:gamma}. Then,
\begin{align}\label{eq:lem12intermediate}
    \E[D^{(t+1)}|\alpha\ttt]-D^{(t)}
    = ~~&\E[A^{(t+1)}-A^{(t)}|\alpha\ttt] + \gamma\E[B^{(t+1)}-B^{(t)}|\alpha\ttt]\notag\\
    \overset{\mathclap{\eqref{eq: E[A]}, \eqref{eq: E[B]}}}{=} ~~&- \theta A^{(t)}+ \tfrac{\theta}{n}\sum_{i=1}^n\left((u_i^{(t)}-\alpha_i^*)^2 - \Big(1- \tfrac{\theta}{p_i}\Big)(\kappa_i^{(t)})^2\right)\notag\\
    &+ \gamma \bigg(-\tfrac{2 \theta }{\lambda}\nabla P(w^{(t)})^T(w^{(t)}-w^*) + \sum_{i=1}^n \frac{\theta^2v_i}{n^2 \lambda^2p_i}(\kappa^{(t)}_i)^2\bigg)\notag\\
    \overset{\mathclap{\eqref{eq:C1},\eqref{eq:C2}}}{\leq} ~~& -\theta A^{(t)} + \mathbf{C_1} + \mathbf{C_2}.
\end{align}

Now,
\begin{eqnarray}\label{eq:intresult}
    \mathbf{C_1}&\overset{\eqref{eq:C1}}{=}&\frac{\theta}{n}\sum_{i=1}^n(u_i^{(t)}-\alpha_i^*)^2-\tfrac{2 \gamma\theta}{\lambda}\nabla P(w^{(t)})^T(w^{(t)}-w^*)\notag\\
    &\overset{\mathclap{(\ref{eq: strong ineq})}}{\leq} &2\theta\tilde{L}\left(P(w^{(t)})-P(w^*)-\tfrac{\lambda}{2}\norm{w^{(t)}-w^*}^2\right) - \tfrac{2 \gamma\theta}{\lambda}\nabla P(w^{(t)})^T(w^{(t)}-w^*)\notag\\
    &\overset{\mathclap{\gamma = \lambda \tilde{L}}}{=}&-\gamma \theta\norm{w^{(t)}-w^*}^2 + 2\theta\tilde{L}\left(P(w^{(t)})-P(w^*)-\nabla P(w^{(t)})^T(w^{(t)}-w^*)\right)\notag\\
    &\leq& -\gamma\theta\norm{w^{(t)}-w^*}^2,
\end{eqnarray}
where the last inequality follows from convexity of $P(w)$, i.e., $P(w^{(t)})-P(w^*) \leq \nabla P(w^{(t)})^T(w^{(t)}-w^*)$. Combining \eqref{eq:lem12intermediate} and \eqref{eq:intresult} gives
\begin{eqnarray*}
    \E[D^{(t+1)}|\alpha\ttt]-D^{(t)} \leq -\theta A^{(t)} -\gamma\theta\norm{w^{(t)}-w^*}^2 + \mathbf{C_2} = -\theta D\ttt+ \mathbf{C_2}.
\end{eqnarray*}
Rearranging gives the result.\qed
\end{proof}

\begin{proof}[Proof of Lemma~\ref{thm: average convex}]
For this result we assume that the average of the loss functions $\frac1n\sum \phi_i(\cdot)$ is convex. Note that one can define parameters $ \mathbf{\bar C_1}$ and $\mathbf{\bar C_2}$ that are analogous to $\mathbf{C_1}$ and $\mathbf{C_2}$ in \eqref{eq:C1} and \eqref{eq:C2} but with $\gamma$ replaced by $\bar\gamma$. Then, the same arguments as those used in \eqref{eq:lem12intermediate} can be used to show that
\begin{equation}\label{eq:Ddiffbar}
  \E[\bar D^{(t+1)}|\alpha\ttt]-\bar D^{(t)}  \leq -\theta A^{(t)} + \mathbf{\bar C_1} + \mathbf{\bar C_2}.
\end{equation}
Now, note that by Lipschitz continuity of $\phi'(\cdot)$ one has
\begin{equation}\label{eq:Lipschitz}
    (u_i^{(t)}-\alpha_i^*)^2 = \left(\phi'_i(x_i^Tw)-\phi'_i(x_i^Tw^{(t)})\right)^2 \leq L_i^2\norm{w^*-w^{(t)}}^2.
\end{equation}

Further, since the average of the losses is convex, $P(w)$ is strongly convex, so
\begin{equation} \label{eq: strong convexity}
    P(w^*) - P(w^{(t)}) \geq \nabla P(w^{(t)})^T(w^*-w^{(t)}) + \tfrac{\lambda}{2}\norm{w^*-w^{(t)}}^2
\end{equation}
and since $w^*$ is the minimizer
\begin{equation}\label{eq: minimzer}
    P(w^{t}) - P(w^*) \geq \tfrac{\lambda}2\norm{w^{(t)}-w^*}^2.
\end{equation}
Now, adding \eqref{eq: strong convexity} and \eqref{eq: minimzer} gives
\begin{equation}\label{eq:gradPnormw}
    \nabla P(w^{(t)})^T(w^{(t)}-w^*) \geq \lambda\norm{w^{(t)}-w^*}^2.
\end{equation}
Therefore,
\begin{align}\label{eq:barC1vsw}
    \mathbf{\bar C_1}=~~&\frac{\theta}{n}\sum_{i=1}^n(u_i^{(t)}-\alpha_i^*)^2-\frac{2 \bar\gamma\theta}{\lambda}\nabla P(w^{(t)})^T(w^{(t)}-w^*)\notag\\
    \overset{\eqref{eq:Lipschitz},\eqref{eq:gradPnormw}}{\leq} ~~& \frac{\theta}{n}\sum_{i=1}^nL_i^2\norm{w^{(t)}-w^*}^2- 2\bar\gamma \theta\norm{w^{(t)}-w^*}^2\notag\\
    \overset{\eqref{eq:bargamma}}{\leq}~~& -\bar\gamma\theta\norm{w^{(t)}-w^*}^2.
\end{align}

Thus, from \eqref{eq:Ddiffbar} and \eqref{eq:barC1vsw} we have that $\E[D^{(t+1)}|\alpha\ttt]-D\ttt  \leq -\theta D^{(t)} + {\mathbf{\bar C_2}}$, which is the desired result.\qed
\end{proof}

\subsection{Proof of Lemma \ref{lem: optimal probabilities}}
\begin{proof}
    This is easy to verify by derive KKT conditions of optimization problem \eqref{Prob: to derive best probabilities}, which is
    \begin{equation}
    \begin{cases}
            % \frac{n\lambda^2\sum_{i\in I_ \kappa}\kappa_i^2}{\sum_{i\in I_ \kappa}(n\lambda^2+v_i\gamma)p_i^{-1}\kappa_i^2},\\
             -(n\lambda^2\sum_{i\in I_ \kappa}\kappa_i^2)(\sum_{i\in I_ \kappa}(n\lambda^2+v_i\gamma)p_i^{-1}\kappa_i^2)^{-2}(-(n\lambda^2+v_i\gamma)p_i^{-2}\kappa_i^2) + \mu &= 0 \notag\\
             \sum_{i\in I_k}p_i &= 1\notag\\
    \end{cases}
    \end{equation}

    where $\mu$ is the Lagrange multiplier.

    Therefore, we have
    \begin{equation}
        \frac{p_i}{p_j}= \frac{\sqrt{n\lambda^2+v_i\gamma}|\kappa_i|}{\sqrt{n\lambda^2+v_j\gamma}|\kappa_j|}, \mbox{  for all  } i, j\in I_\kappa.
    \end{equation}

    Considering $\sum_{i\in I_\kappa}p_i=1$, we show that the optimal probabilities \eqref{eq: optimal probabilities}. \eqref{eq: optimal theta} can be further derived by combine \eqref{eq: optimal probabilities} and \eqref{Prob: to derive best probabilities}.
\end{proof}

\subsection{Proof of Theorems~\ref{cor: all convex} and \ref{cor: aver-convex}}
\begin{proof}[Proof of Theorem%\section{Proof of Lemma \ref{lem: optimal probabilities}}
%This is easy to verify by derive KKT conditions of optimization problem \eqref{Prob: to derive best probabilities}.
~\ref{cor: all convex}]
Note that substituting $p^*$ (where $p^*$ is defined in Lemma~\ref{lem: optimal probabilities}) into $\Theta(\kappa, p^*)$ in \eqref{eq:Theta} and using the Cauchy-Schwartz inequality, gives
\begin{align}
    \Theta(\kappa, p^*) =  \frac{n \lambda^2 \sum_{i\in I_ \kappa} \kappa_i^2}{(\sum_{i\in I_ \kappa}\sqrt{v_i \gamma + n \lambda^2}|\kappa_i|)^2}
    =\frac{n \lambda^2 \sum_{i=1}^n \kappa_i^2}{(\sum_{i=1}^n\sqrt{v_i \gamma + n \lambda^2}|\kappa_i|)^2} \geq \frac{n \lambda^2}{\sum_{i=1}^n(v_i \gamma + n \lambda^2)} \overset{\eqref{eq:thetastar}}{=} \theta^*.
\end{align}
The above confirms that $\theta^*$ in \eqref{eq:Theta} is a (constant) global lower bound of $\Theta(\kappa, p^*)$ at every iteration.
Thus, using the arguments following Lemma~\ref{thm: seperate convex}, setting $p^{(t)} = p^*$ (as computed in Lemma~\ref{lem: optimal probabilities}) at each iteration gives
    \begin{equation}\label{eq:Ddiffthetastar}
        \E\left[D^{(t+1)}|\alpha\ttt\right] \leq (1- \theta^*)D\ttt.
    \end{equation}
    That is, \eqref{eq:Dreduce} used with $\theta\equiv \theta^*$ holds. Because \eqref{eq:Ddiffthetastar} holds at every iteration of
    Algorithm~\ref{Alg: adfSDCA}, one can show that
    \begin{equation}\label{eq:DC0}
    \E\left[D\ttt\right] \leq (1- \theta^*)^tC_0 \leq e^{- \theta^*t}C_0,
\end{equation}
where $C_0$ is defined in \eqref{eq:C0}. Now, note that $P(w)$ is $(L+\lambda)$-smooth, i.e.,  $P(w) - P(w^*) \leq \tfrac{\lambda + L}{2}\norm{w-w^*}^2$, so
\begin{eqnarray*}
  D\ttt = \tfrac1n \|\alpha\ttt - \alpha^*\|^2 + \gamma\|w\ttt - w^*\|^2
  \geq \gamma\|w\ttt - w^*\|^2
  \geq \tfrac{2\gamma}{\lambda+L}(P(w\ttt) - P(w^*)).
\end{eqnarray*}
This means that we must find $T$ for which
\begin{equation}\label{eq:PdiffC0}
 \E[P(w^{(T)}) - P(w^*)]\leq \tfrac{\lambda+L}{2\gamma}e^{- \theta^*T}C_0 \leq \epsilon.
\end{equation}
Subsequently, the expression for $T$ in \eqref{eq:Tconvex} is obtained by multiplying through by $e^{\theta^*T}/\epsilon$, taking natural logs, rearranging and noting that
\begin{equation*}
  \frac1{\theta^*} = \frac{\sum_{i=1}^n (v_i \gamma + n\lambda^2)}{n\lambda^2} = n + \frac{\gamma}{n\lambda^2}\sum_{i=1}^n v_i \overset{\eqref{eq:gamma}}{=} n + \frac{\tilde L}{n\lambda}\sum_{i=1}^n v_i\overset{\eqref{eq:MQ}}{=} n + \frac{\tilde L Q}{\lambda}.
\end{equation*}
\end{proof}

\begin{proof}[Proof of Theorem~\ref{cor: aver-convex}]
Here we assume that the average loss $\frac1n\sum_{i=1}^n\phi_i(\cdot)$ is convex, but that individual loss functions $\phi_i(\cdot)$ may not be. The proof of this result is almost identical to the proof of Theorem~\ref{cor: all convex}, but with the parameters defined in Section~\ref{sec:analysisCaseII}. Similarly to \eqref{eq:PdiffC0} we must find $T$ for which
\begin{equation}\label{eq:barPdiffC0}
 \E[P(w^{(T)}) - P(w^*)]\leq \tfrac{\lambda+L}{2\bar\gamma}e^{- \theta^*T}\bar C_0 \leq \epsilon,
\end{equation}
where $\bar\gamma = \frac{1}{n}\sum_{i=1}^nL_i^2$ is defined in \eqref{eq:bargamma} and $\bar C_0$ is defined in \eqref{eq:barC0}. The expression $T$ in \eqref{eq:Tavconvex} is obtained by multiplying through by $e^{\theta^*T}/\epsilon$, taking natural logs, rearranging and noting that
\begin{equation*}
  \frac1{\theta^*} = \frac{\sum_{i=1}^n (v_i \bar\gamma + n\lambda^2)}{n\lambda^2} = n + \frac{\bar\gamma}{\lambda^2}\Big(\frac1n\sum_{i=1}^n v_i\Big)\overset{\eqref{eq:MQ}}{=} n + \frac{\bar\gamma Q}{\lambda^2}.
\end{equation*}
\qed
\end{proof}

\subsection{Proof of Corollary \ref{variance reduction}}

\begin{proof}
    Recall that $w^*$ denotes the minimizer of \eqref{Prob: L2EMR} and $\alpha_i^* = -\phi'(x_i^Tw^*)$. Let Assumption~\ref{ass: smooth} hold. Then

    \begin{eqnarray}\label{cor:takeexp}
         \Big\|\frac{1}{np_i}\kappa_i^{(t)}x_i\Big\|^2
       & \overset{\eqref{eq:viQ}}{=}  &\frac{1}{n^2p_i^2} (\kappa_i^{(t)})^2v_i\notag\\
         &\overset{\mathclap{\eqref{eq: optimal probabilities}}}{=}  & \frac{1}{n^2}\left(\frac{\sum_{j\in I_ \kappa}\sqrt{n\lambda^2+v_j\gamma}|\kappa_j^{(t)}|}{\sqrt{n\lambda^2+v_i\gamma}|\kappa_i^{(t)}|}\right) ^2(\kappa_i^{(t)})^2v_i\notag\\
       &\overset{\rm(CS)}{\leq} &\frac{1}{n^2}\frac{\sum_{j=1}^n(n\lambda^2+v_j\gamma)\sum_{j=1}^n(\kappa_j^{(t)})^2}{(n\lambda^2+v_i\gamma)(\kappa_i^{(t)})^2} (\kappa_i^{(t)})^2v_i\notag\\
       &  =  &\frac{\sum_{j=1}^n(n\lambda^2+v_j\gamma)}{n^2(n\lambda^2+v_i\gamma)}\norm{\kappa^{(t)}}^2v_i \notag\\
       & \overset{\eqref{eq:MQ}}{=}  &\frac{n^2\lambda^2+\gamma n Q}{n^2(n\lambda^2+v_i\gamma)}\norm{\kappa^{(t)}}^2v_i.
    \end{eqnarray}
Taking the (conditional) expectation of \eqref{cor:takeexp} gives

\begin{eqnarray}\label{eq:expof}
   \E\left[ \Big\|\frac{1}{np_i} \vc{\kappa_i}{t}x_i\Big\|^2 \,\Big|\alpha^{(t-1)}\right]
    &=  &\sum_{i=1}^np_i\left(\frac{n^2\lambda^2+\gamma n Q}{n^2(n\lambda^2+v_i\gamma)}\norm{\kappa^{(t)}}^2v_i\right)\notag\\
    &\leq & \sum_{i=1}^n\left(\frac{n^2\lambda^2+\gamma n Q}{n^2(n\lambda^2+v_i\gamma)}\norm{\kappa^{(t)}}^2v_i\right)\notag\\
    &\leq & \sum_{i=1}^n\left(\frac{n^2\lambda^2+\gamma n Q}{n^3\lambda^2}\norm{\kappa^{(t)}}^2v_i\right)\notag\\
    &=&\left(\frac{n^2\lambda^2+\gamma n Q}{n^2\lambda^2}\norm{\kappa^{(t)}}^2\right)\left(\frac{1}{n}\sum_{i=1}^nv_i\right)\notag\\
    &\overset{\eqref{eq:MQ}}{=} &Q\left(1+\frac{\gamma Q}{n \lambda^2}\right) \norm{\kappa^{(t)}}^2.
\end{eqnarray}

Finally
\begin{align*}
    \norm{\kappa^{(t)}}^2 = \E\left[\norm{\kappa^{(t)}}^2|\alpha^{(t-1)}\right]
    =~~& \E\Big[\sum_{i=1}^n\Big(\alpha^{(t)}_i + \phi'_i(x_i^Tw^{(t)})\Big)^2|\alpha^{(t-1)}\Big]\notag\\
    =~~& \E\Big[\sum_{i=1}^n\Big(\alpha^{(t)}_i - \alpha^* - \phi'_i(x_i^Tw^*)+ \phi'_i(x_i^Tw^{(t)})\Big)^2|\alpha^{(t-1)}\Big]\notag\\
     \leq~~& 2\E[\norm{\alpha^{(t)}-\alpha^*}^2|\alpha^{(t-1)}] + 2L\E[\norm{w^{(t)}-w^*}^2|\alpha^{(t-1)}].
\end{align*}
Combining the last step with \eqref{eq:expof} gives the result. \qed
\end{proof}

The proof of Corollary~\ref{variance reduction_nonconvex} is essentially identical, but with the notation established in Section~\ref{sec:analysisCaseII}, so we omit it for brevity.

\subsection{Proof of Theorems \ref{thm: minibatch-all convex} and \ref{thm: minibatch-seprate convex}}

Recall that $A^{(t)}$ and $B^{(t)}$ are defined in \eqref{eq:AB}. To prove Theorem \ref{thm: minibatch-all convex} we need the following two conditions to hold,
    \begin{align}
                \E_{\hat{S}}\left[A^{(t+1)} - A^{(t)}|\alpha\ttt\right] =& -\theta A^{(t)} + \frac{\theta}{n}\sum_{i=1}^n\left((u_i^{(t)}-\alpha_i^*)^2 - \left(1- \tfrac{\theta}{bp_i}\right)(\kappa^{(t)}_i)^2\right),\label{eq: ES[A]}\\
        \E_{\hat{S}}\left[B^{(t+1)} - B^{(t)}|\alpha\ttt\right]  \leq& -\frac{2 \theta }{\lambda}\nabla P(w^{(t)})^T(w^{(t)}-w^*) + \sum_{i=1}^n \frac{\theta^2v_i(\kappa^{(t)}_i)^2}{n^2 \lambda^2bp_i}.\label{eq: ES[B]}
    \end{align}
Note that $\E_{\hat{S}}\left[A^{(t+1)}-A^{(t)}|\alpha\ttt\right] = \sum_{i=1}^nbp_i(A^{(t+1)}-A^{(t)})$, and so \eqref{eq: ES[A]} is obtained by using arguments similar to those used in the proof of \eqref{eq: E[A]}. To show \eqref{eq: ES[B]}, first we have
    \begin{align}
    B^{(t+1)}-B^{(t)} &=  \norm{w^{(t+1)}-w^*}^2 - \norm{w^{(t)}-w^*}^2\notag\\
                            &= \norm{w^{(t)} - \sum_{i\in S}\frac{\theta}{n\lambda bp_i}\kappa_i^{(t)}x_i^T -w^*}^2 - \norm{w^{(t)}-w^*}^2\notag\\
                            & = -\frac{2 \theta}{n\lambda}\sum_{i\in S}\frac{\kappa^{(t)}_i}{bp_i}x_i^T(w^{(t)}-w^*) + \frac{\theta^2}{n^2\lambda ^2}\norm{\sum_{i\in S}\frac{\kappa_i^{(t)}}{bp_i}x_i}^2.
    \end{align}

    Therefore, we have
    \begin{align}
        \E_{\hat{S}}\left[B^{(t+1)}-B^{(t)}|\alpha\ttt\right] &= \E_{\hat{S}}\bigg[-\frac{2 \theta}{n\lambda}\sum_{i\in S}\frac{\kappa^{(t)}_i}{bp_i}x_i^T(w^{(t)}-w^*) + \frac{\theta^2}{n^2\lambda ^2}\norm{\sum_{i\in S}\frac{\kappa_i^{(t)}}{bp_i}x_i}^2|\alpha\ttt\bigg]\notag\\
        & = -\frac{2 \theta}{n \lambda}\sum_{i=1}^n\kappa_i^{(t)}x_i^T(w^{(t)}-w^*) + \frac{\theta^2}{n^2 \lambda^2}\E_{\hat{S}}\norm{\sum_{i\in S}\frac{\kappa_i^{(t)}}{bp_i}x_i}^2.
    \end{align}

    Note that from Section \ref{sec:eso} we have
    \begin{equation}\label{eq: theorem5}
        \E_{\hat{S}}\norm{\sum_{i\in S}\frac{\kappa_i^{(t)}}{bp_i}x_i}^2 \leq \sum_{i=1}^nbp_iv_i'\bigg(\frac{\kappa_i^{(i)}}{bp_i}\bigg)^2 = \sum_{i=1}^n \frac{v_i' (\kappa_i^{(t)})^2}{bp_i},
    \end{equation}

    where $v_i'$ is defined in \eqref{ESO, v}. We can then derive \eqref{eq: ES[B]} by using \eqref{eq: theorem5} and $\nabla P(w^{(t)}) =\frac{1}{n}\sum_{i=1}^n \kappa_i^{(t)}x_i$.

\begin{proof}[Proof of Theorem \ref{thm: minibatch-all convex}]
Define
\begin{equation}\label{eq:C}
  \mathbf{C}(\theta, p^{(t)}, \kappa^{(t)} ) \eqdef \sum_{i=1}^n\left(-\frac{\theta}{n}\Big(1- \frac{\theta}{bp_i}\Big)+\frac{\theta^2v_i'\gamma}{n^2\lambda^2bp_i}\right)(\kappa_i^{(t)})^2.
\end{equation}
Then
    \begin{eqnarray}
        \E_{\hat{S}}[D^{(t+1)} - D^{(t)}|\alpha\ttt] &=& \E_{\hat{S}}[A^{(t+1)}-A^{(t)}|\alpha\ttt] + \gamma \E_{\hat{S}}[B^{(t+1)}-B^{(t)}|\alpha\ttt]\notag\\
        &\overset{\mathclap{\eqref{eq: ES[A]}, \eqref{eq: ES[B]}}}{\leq}~~&  -\theta A^{(t)} + \frac{\theta}{n}\sum_{i=1}^n\left(\Big(u_i^{(t)}-\alpha_i^*)^2 - (1- \frac{\theta}{bp_i}\Big)(\kappa^{(t)}_i)^2\right)\notag\\
        && \qquad + \gamma \bigg(-\frac{2 \theta }{\lambda}\nabla P(w^{(t)})^T(w^{(t)}-w^*) + \sum_{i=1}^n \frac{\theta^2v_i'(\kappa^{(t)}_i)^2}{n^2 \lambda^2bp_i}\bigg)\notag\\
        &\overset{\mathclap{\eqref{eq:intresult}}}{\leq}& -\theta A^{(t)} - \theta\gamma\norm{w^{(t)}-w^*}^2) + \mathbf{C}(\theta, p^{(t)}, \kappa^{(t)})\notag\\
        &=&-\theta D^{(t)} + \mathbf{C}(\theta, p^{(t)}, \kappa^{(t)} ).
    \end{eqnarray}

    We can then derive the optimal probabilities to ensure that $\mathbf{C}(\theta, p^{(t)}, \kappa^{(t)} )\leq 0$, i.e.,
    \begin{equation}\label{eq: optimal batch theta}
        \theta \leq \Theta( p^{(t)}, \kappa^{(t)}) := \frac{n\lambda^2b \sum_{i\in I_(\kappa^{(t)})} (\kappa_i^{(t)})^2}{\sum_{i\in I_{\kappa^{(t)}}}(n\lambda^2+v_i\gamma)(p_i^{(t)})^{-1}(\kappa_i^{(t)})^2}
    \end{equation}

    and then making $\theta$ as large as possible. Indeed, to have largest $\theta$ we arrive at the same optimal probabilities as in Lemma \ref{lem: optimal probabilities}. Using these optimal probabilities we find a fixed $\theta^*$ such that
    \begin{equation}
         \theta^* \eqdef \frac{n\lambda^2b}{\sum_{i=1}^n(n\lambda^2+v_i\gamma)}.
     \end{equation}
    Furthermore, the complexity result in this mini-batch setting follows: $\E[P(w^{t})-P(w^*)]\leq \epsilon$ holds if
    \begin{equation}
        T\geq \bigg(\frac{n}{b} + \frac{\tilde{L}Q'}{b\lambda}\bigg)\log\bigg(\frac{(\lambda +L)C_0}{\lambda \tilde{L} \epsilon}\bigg).
    \end{equation}
\end{proof}

\begin{proof}[Proof of Theorem~\ref{thm: minibatch-seprate convex}]
Define
\begin{equation}\label{eq:Cbar}
  \mathbf{\bar C}(\theta, p^{(t)}, \kappa^{(t)} ) \eqdef \sum_{i=1}^n\left(-\frac{\theta}{n}\Big(1- \frac{\theta}{bp_i}\Big)+\frac{\theta^2v_i'\bar\gamma}{n^2\lambda^2bp_i}\right)(\kappa_i^{(t)})^2.
\end{equation}
Now
    \begin{eqnarray}
        \E_{\hat{S}}[\bar D^{(t+1)} - \bar D^{(t)}|\alpha\ttt] &=& \E_{\hat{S}}[A^{(t+1)}-A^{(t)}|\alpha\ttt] + \bar \gamma \E_{\hat{S}}[B^{(t+1)}-B^{(t)}|\alpha\ttt]\notag\\
        &\overset{\mathclap{\eqref{eq: ES[A]}, \eqref{eq: ES[B]}}}{\leq}~~&  -\theta A^{(t)} + \frac{\theta}{n}\sum_{i=1}^n\left(\Big(u_i^{(t)}-\alpha_i^*)^2 - (1- \frac{\theta}{bp_i}\Big)(\kappa^{(t)}_i)^2\right)\notag\\
        && \qquad + \bar \gamma \bigg(-\frac{2 \theta }{\lambda}\nabla P(w^{(t)})^T(w^{(t)}-w^*) + \sum_{i=1}^n \frac{\theta^2v_i'(\kappa^{(t)}_i)^2}{n^2 \lambda^2bp_i}\bigg)\notag\\
        &\overset{\mathclap{\eqref{eq:barC1vsw}}}{\leq}& -\theta A^{(t)} - \theta\bar\gamma\norm{w^{(t)}-w^*}^2) + \mathbf{\bar C}(\theta, p^{(t)}, \kappa^{(t)})\notag\\
        &=&-\theta D^{(t)} + \mathbf{\bar C}(\theta, p^{(t)}, \kappa^{(t)} ).
    \end{eqnarray}
Similar arguments to those made in the final stages of the proof of Theorem~\ref{thm: minibatch-seprate convex} can be used to show that if $T$ is given by the expression in \eqref{eq:minibatchnonconvexT} then $\E[P(w^{t})-P(w^*)]\leq \epsilon$. \qed
\end{proof}

\end{document}